%% file: diff-Ramsey.tex
\documentclass[12pt]{amsart}

\numberwithin{equation}{section}
\newcommand*{\LatexDef}{.}
\input{"\LatexDef/def.tex"}

\usepackage[all]{xy} 

\makeatletter
\let\stdsubsubsection\subsubsection
\renewcommand*\subsubsection{\addtocounter{equation}{1}\stdsubsubsection}
\makeatother

\makeatletter
\let\stdpart\part
\renewcommand*\part{%
	\@ifstar{\starpart}{\@dblarg\nostarpart}}
\newcommand*\starpart[1]{\clearpage {\Large \stdpart*{#1}} \bigskip}
\def\nostarpart[#1]#2{\clearpage {\Large \stdpart[{#1}]{#2}} \vspace{1.6cm}}
\makeatother

\let\oldtocpart=\tocpart
\let\oldtocsection=\tocsection
\let\oldtocsubsection=\tocsubsection
\renewcommand{\tocpart}[2]{\hspace{0em} \oldtocpart{#1}{#2}}
\renewcommand{\tocsection}[2]{\hspace{0em} \sc \small \oldtocsection{#1}{#2}}
\renewcommand{\tocsubsection}[2]{\hspace{3em} \oldtocsubsection{#1}{#2}}

\newcommand{\instanceref}[2]{Example \exampleslabelref{#1}{#2}}

\crefname{instances}{Instances}{Instances}
\newenvironment{instances}[1][a]{\refstepcounter{equation}\par\medskip\noindent \textbf{Instances~\theequation.} 
	\vspace{.6em}
	\begin{enumerate}[\bfseries (#1)]
		\setlength{\itemsep}{.8em}
		\rmfamily}{\end{enumerate}\smallskip}

\title[]{A Ramsey theorem on semigroups and a general van der Corput lemma}
\author[]{Anush Tserunyan}
\address{Department of Mathematics, University of Illinois at Urbana-Champaign, IL, 61801, USA}
\email{anush@illinois.edu}

\date{}
\newcommand{\T}[2][\LC,\LC']{T(#2 / #1)}

\newcommand{\FP}{\text{\textnormal{FP}}}
\newcommand{\IP}{\text{\textnormal{IP}}\xspace}
\newcommand{\IPstar}{{\text{\textnormal{IP}}^\ast}}

\newcommand{\Fh}{\FC_h}

\newcommand{\m}{\bar{m}}
\newcommand{\n}{\bar{n}}

\newcommand{\dualF}{\FC^0}
\newcommand{\simF}{\sim_{\FC}}

\newcommand{\tilA}[1][A]{\tilde{#1}}

\newcommand{\ext}{\text{\textnormal{ext}}}
\newcommand{\depth}{\text{\textnormal{depth}}}
\newcommand{\Depth}{\text{\textnormal{Depth}}}

\newcommand{\updens}{{\overline{d}}}

\newcommand{\asubset}[1][\FC]{\subseteq_{#1}}
\newcommand{\asupset}[1][\FC]{\supseteq_{#1}}

\newcommand{\thickL}[2][\LC]{$\del(#2 / #1)$-thick}
\newcommand{\DthickL}[2][\LC]{$\De(#2 / #1)$-thick}

\newcommand{\thick}[2][\FC]{$\del_{#1}(#2)$-thick}
\newcommand{\Dthick}[2][\FC]{$\De_{#1}(#2)$-thick}

\newcommand{\pos}[1][\FC]{{{#1}^\plus}}
\newcommand{\Pos}[1][\FC]{\displaystyle \pos[#1]}

\newcommand{\D}{\straighttext{D}\xspace}
\newcommand{\M}{\straighttext{M}\xspace}
\newcommand{\LL}{\straighttext{L}\xspace}

\newcommand{\MeasF}[1][\FC]{\PC_{#1}}

\begin{document}

\begin{abstract}	
	A major theme in arithmetic combinatorics is proving multiple recurrence results on semigroups (such as \Szemeredi's theorem) and this can often be done using methods of ergodic Ramsey theory. What usually lies at the heart of such proofs is that, for actions of semigroups, a certain kind of one recurrence (mixing along a filter) amplifies itself to multiple recurrence. This amplification is proved using a so-called van der Corput difference lemma for a suitable filter on the semigroup. Particular instances of this lemma (for concrete filters) have been proven before (by Furstenberg, Bergelson--McCutcheon, and others), with a somewhat different proof in each case. We define a notion of differentiation for subsets of semigroups and isolate the class of filters that respect this notion. The filters in this class (call them $\del$-filters) include all those for which the van der Corput lemma was known, and our main result is a van der Corput lemma for $\del$-filters, which thus generalizes all its previous instances. This is done via proving a Ramsey theorem for graphs on the semigroup with edges between the semigroup elements labeled by their ratios.
\end{abstract}

\maketitle

\tableofcontents

\section{Introduction}

The current paper concerns a generalization of certain types of lemmas, known as van der Corput difference lemmas\footnote{The name comes from the well-known van der Corput difference theorem proved by Johannes van der Corput in \cite{vdC}.}, that are used in proving multiple recurrence results in ergodic Ramsey theory. In this section, we describe the general context in which these lemmas are applied, using the famous Furstenberg Multiple Recurrence theorem as a motivating example. Furthermore, we state our generalization of these lemmas, and conclude this section with a discussion of a Ramsey-type theorem (our main result), whose immediate application gives the mentioned generalization.

\subsection{Multiple recurrence}

One of the main themes in arithmetic combinatorics is proving multiple recurrence results for a given semigroup. Such is the celebrated \Szemeredi's theorem:

\begin{theorem}[\Szemeredi, \cite{Szemeredi}]
	Any subset $A \subseteq \N$ of positive upper density, i.e. $\updens(A) := \limsup_{n \to \w} {|A \cap [0,n)| \over n} > 0$, contains arbitrarily long arithmetic progressions. In other words, for every $k \ge 1$, there is $n \in \N$ such that
	$$
	A \cap (A - n) \cap (A - 2n) \cap \cdots \cap (A - kn) \ne \0.
	$$
\end{theorem}

The conclusion of \Szemeredi's theorem can be viewed as a multiple recurrence statement for the action of the semigroup $G = \N$ on itself by right translation, where we equip the action space $\N$ with the upper density function viewed as a finitely subadditive invariant probability measure. Shortly after \Szemeredi's original proof, Furstenberg came up with a way of translating this statement to a multiple recurrence statement for an actual probability measure preserving (p.m.p.) action of $\Z$ (with the measure being countably additive), and then proved the latter statement (now known as Furstenberg's Multiple Recurrence Theorem \cite{Furst_orig}) for arbitrary p.m.p. actions of $\Z$.

\subsection{Mixing along filters}

One of the key ingredients in the proof of the Multiple Recurrence Theorem is the fact that a certain strong (quantitative) one recurrence, known as \emph{weak mixing}, amplifies itself to a strong multiple recurrence (weak mixing of all orders). This amplification is where the mentioned van der Corput difference lemmas are used:

\begin{displaymath}\xymatrixcolsep{9pc}\xymatrix@C-.2cm{
		\text{strong one recurrence} \ar@{=>}[r]^-{\txt{\scriptsize van der Corput trick}} & \text{strong multiple recurrence.}
	}\end{displaymath}

These strong notions of recurrence as well as the van der Corput involve a filter $\FC$ on the acting group or, more generally, semigroup $G$. The definitions of filters, limits along them, and other related terminology, is given in \cref{subsec:filters} below. Here is a typical example to keep in mind:
\begin{example}\label{example:density_filter}
	For $G = \N$ (or any amenable semigroup), define \emph{upper density} $\updens(A)$ for subsets $A \subseteq \N$ by
	$$
	\updens(A) := \limsup_{n \to \w} {|A \cap [0,n)| \over n}.
	$$
	The sets of (upper) density $0$ form an ideal, so their complements form a filter, which we denote by $\FC_\updens$ and refer to as the \emph{density filter} on $\N$.
\end{example}

\begin{defn}
Let $G$ be a semigroup, $\FC$ a filter on $G$, and $(X,\nu)$ a probability space. Measure-preserving (right) action $(X,\nu) \ractson^\al G$ is called mixing along $\FC$ if for every $\nu$-measurable $A,B \subseteq X$, we have
$$
\lim_{g \to \FC} \nu(A \cap B \cdot_\al g^{-1}) = \nu(A) \nu(B).
$$
\end{defn}
Because $\nu(B \cdot_\al g^{-1}) = \nu(B)$, what this definition says is that as $g \to \FC$, the sets $A$ and $B \cdot_\al g^{-1}$ become more and more probabilistically independent. In other words, for any pair $(x,y) \in X^2$, the pair $(x, y \cdot_\al g)$ looks more and more like a random pair $(u,v) \in X^2$.

\begin{remark}
	We can recover the usual notions of mixing by choosing appropriate filters: the density filter $\FC_d$ for weak mixing \cite{Berg-Gorod}*{Theorem 1.1}, the filter $\IPstar$ for mild mixing \cite{Furst_book}*{Proposition 9.22}, and the \Frechet filter for strong mixing \cite{Rudolph}*{Definition 4.3}.
\end{remark}

\subsection{The van der Corput property}

Any probability measure-preserving (right) action \\ $(X,\nu) \ractson^\al G$ of a semigroup $G$ can be lifted to a unitary (left) action $G \actson^\al L^2(X,\nu)$ by $(g \cdot_\al f)(x) := f(x \cdot_\al g)$. In terms of this unitary action, denoting by $\gen{\cdot,\cdot}$ the inner product in $L^2(X,\nu)$, mixing along $\FC$ is equivalent to the following: for every $f_0,f_1 \in L^2(X,\nu)$,
$$
\lim_{g \to \FC} \gen{f_0, g \cdot_\al f_1} = \int_X f_0 d\nu \int_X f_1 d\nu.
$$
In light of this, putting $e_g := g \cdot_\al f_1$, one can see how the following property of filters on semigroups may be relevant here:

\begin{defn}\label{defn:vdC-prop}
	A filtered \M-semigroup $(G, \PC, \FC)$ (see \cref{subsec:M-semigroups} for the definition) is said to have the \emph{van der Corput property} if for every weakly upper $\PC$-semimeasurable (see \cref{defn:M-measurability}) bounded sequence $(e_g)_{g \in G}$ in a Hilbert space $\Hil$, we have
	$$
	\lim_{h \to \FC} \lim_{g \to \FC} \gen{e_g, e_{gh}} = 0 \implies \lim_{g \to \FC} \gen{f,e_g} = 0, \ \forall f \in \Hil.
	$$
\end{defn}

The conclusion in the implication above simply says that the limit along $\FC$ of the sequence $(e_g)_{g \in G}$ is $0$ in the weak topology of $\Hil$. This property is really a stronger and more general version of the simple Hilbert space fact (consequence of Bessel's inequality) that any bounded sequence of pairwise orthogonal vectors converges to $0$ in the weak topology, where the convergence is in the usual sense, i.e. along the \Frechet filter (see \examplesref{examples:filters}{example:Frechet_filter} for the definition). Here, we write this fact in the appropriate form to make the similarity apparent: 

\begin{lemma}\label{st:Bessel}
	For every bounded sequence $(e_n)_{n \in \N}$ in $\Hil$,
	$$
	\forall m \ne 0 \forall n \; \gen{e_n, e_{n+m}} = 0 \implies \lim_{n \to \w} \gen{f, e_n} = 0, \ \forall f \in \Hil.
	$$
\end{lemma}

\begin{remark}
	Letting $\FC$ denote the \Frechet filter on $\N$, our main result (\cref{st:difference-Ramsey_for_del-filters}) implies that $(\N, \dualF \uplus \FC, \FC)$ has the van der Corput property, as explained in \cref{remark:generalizing_Bessel_for_Frechet}. However, $(\N, \Pow(\N), \FC)$ does not have the van der Corput property, in other words, the implication
	$$
	\lim_{m \to \w} \lim_{n \to \w} \; \gen{e_n, e_{n+m}} = 0 \implies \lim_{n \to \w} \gen{f, e_n} = 0, \ \forall f \in \Hil
	$$
	does not hold for \emph{all} sequences $(e_n)_{n \in \N} \subseteq \Hil$. For example, let $(f_n)_{n \in \N}$ be a sequence of orthonormal vectors in a Hilbert space and, for each $n \in \N$, take
	$$
	e_n := \defbycases{f_0}{n = 2^k, \text{ for some } k}{f_n}.
	$$
	Then, the condition $\displaystyle \lim_{n \to \w} \lim_{m \to \w} \gen{e_m, e_{m+n}} = 0$ holds, but $\gen{f_0, e_{2^k}} = 1$ for all $k \in \N$. One reason as to why this sequence is a counterexample is that the set $\set{2^k : k \in \N}$ is not $\w$-differentiable $\mmod{\FC}$ as explained in \examplesref{examples_differentiable_sets}{example_differentiability_for_Frechet}.
\end{remark}

In ergodic Ramsey theory, a \emph{van der Corput lemma} usually refers to a statement that the van der Corput property holds for a filter on a semigroup $G$ with $\PC = \Pow(G)$. An instance of this was proven by Furstenberg \cite{Furst_book}*{Lemma 4.9} for the density filter $\FC_d$ on $G = \N$, with $\PC = \Pow(\N)$, as an important ingredient in his proof of Multiple Recurrence Theorem. Furthermore, instances of the van der Corput lemma for various filters have been used in deriving multiple recurrence results for semigroups other than $\N$.

This apparent usefulness of the van der Corput property makes one wonder for which filters (more precisely, filtered \M-semigroups) it holds. Besides the density filter, it was previously known to hold for the $\IPstar$ filter \cite{Furst_book}*{Lemma 9.24} and idempotent ultrafilters \cite{Berg-McCutch}*{Theorem 2.3} on arbitrary semigroups (with $\PC$ being the powerset). Furthermore, a version of this property was noticed and used by the author in \cite{me_prob_groups} for the filter of conull sets of an invariant probability measure $\mu$ on a group (with $\PC$ being the $\si$-algebra of $\mu$-measurable sets). The proofs of these van der Corput lemmas all follow a general flow, even though different features of the filters are used to run this flow. The current work is devoted to pinning down a general property of filters on semigroups (more precisely, filtered \M-semigroups) that implies the van der Corput property and is satisfied by all of the above-mentioned examples.

\subsection{Underlying Ramsey theory}

Besides the natural urge of trying to find one proof that works for all of the existing instances of the van der Corput lemmas, the author's motivation for the current work was a realization that the van der Corput property is driven by a certain Ramsey-theoretic condition for graphs on semigroups, which we now briefly discuss.

\begin{defn}\label{defn:difference-Ramsey_prop}
	Say that a filtered \LL-semigroup $(G, \LC, \FC)$ (see \cref{defn:filtered_L-space}) has the \emph{difference-Ramsey property} if any graph\footnote{By a \emph{graph} we simply mean a binary relation.} $E \subseteq G^2$ satisfying $\forall^\FC h \forall^\FC g \ E(g,gh)$ contains arbitrarily large complete subgraphs in any $\LC$-large set $A$, i.e. for any $n \in \N$ there is a sequence $(g_i)_{i \le n}$ of elements in $A$ such that $E(g_i, g_j)$ for all $i<j<n$.
\end{defn}

One way to think about it is as follows: label each edge $(g_1, g_2) \in E$ with all possible ratios\footnote{Since $G$ is only a semigroup, there may be more than one ratio or none at all.}, i.e. all $h \in G$ such that $g_1 h = g_2$. Then, the hypothesis reads as follows: for $\FC$-almost every label $h \in H$, $\FC$-almost every vertex $g \in G$ has an outgoing edge in $E$ with that label. This is a way of expressing via the semigroup operation that the graph $E$ has lots of edges. The conclusion is, as expected, that the graph $E$ contains arbitrarily large complete subgraphs, and moreover, these subgraphs can be found ``locally'' in large enough subsets of vertices.

We now prove that, indeed, this Ramsey property implies that of van der Corput. The argument we give here is implicitly present in all known proofs of the van der Corput property for particular examples of filters, however, the difference-Ramsey property had not been explicitly isolated before.

\begin{theorem}\label{st:diff-Ramsey=>vdC}
	If a filtered \LL-semigroup $(G, \LC, \FC)$ \textup{(see \cref{defn:filtered_L-space})} has the difference-Ramsey property, then the filtered \M-semigroup $(G, \dualF \cup \LC, \FC)$ has the van der Corput property.
\end{theorem}
\begin{proof}
	Letting $\PC := \dualF \cup \LC$ and using the notation of \cref{defn:vdC-prop}, we fix a weakly upper $\PC$-semimeasurable sequence $(e_g)_{g \in G} \subseteq \Hil$ with $\|e_g\| \le 1$, and suppose the conclusion fails for a nonzero vector $f \in \Hil$, i.e. there is $\e > 0 $ such that the set $A = \set{g \in G : |\gen{e_g, f}| \ge \e}$ is $\FC$-positive. By the semimeasurability hypothesis on $(e_g)_{g \in G}$, $A \in \PC$, so it must be $\LC$-large. Choose $n \in \N$ so large that $\|f\|^2 < n \e^2 / 2$, and $\de > 0$ so small that $(n-1) \|f\|^2 \de \le \e^2/2$. Applying the difference-Ramsey property to $A$ and the graph $E \subseteq G^2$ defined by
	$$
	E(g_1,g_2) :\iff |\gen{e_{g_1}, e_{g_2}}| \le \de,
	$$
	we get ``too many pairwise almost orthogonal vectors'' over $A$, that is, a sequence $(g_i)_{i < n} \subseteq A$ with $|\gen{e_{g_i}, e_{g_j}}| \le \de$ for all $i<j$. Hence, the proof of Bessel's inequality gives a contradiction:
	\begin{align*}
		0 &\le \|f - \sum_{i<n} \gen{f,e_{g_i}}e_{g_i}\|^2 = \|f\|^2 - 2 \sum_{i<n} |\gen{f, e_{g_i}}|^2 + \sum_{i,j < n}\gen{f, e_{g_i}} \overline{\gen{f,e_{g_j}}}\gen{e_{g_i},e_{g_j}} \\
		&= \|f\|^2 - 2 \sum_{i<n} |\gen{f, e_{g_i}}|^2 + \sum_{i<n} |\gen{f, e_{g_i}}|^2 \cdot \|e_{g_i}\|^2 + \sum_{i,j < n, i \ne j} \|f\| \cdot \|e_{g_i}\| \cdot \|f\| \cdot \|e_{g_j}\| \cdot |\gen{e_{g_i},e_{g_j}}|\\
		&\le \|f\|^2 - \sum_{i<n} |\gen{f, e_{g_i}}|^2 + \sum_{i,j < n, i \ne j} \|f\|^2 \cdot \de \\
		&= \|f\|^2 - n \e^2 +  n (n-1) \cdot \|f\|^2 \cdot \de \\
		&\le \|f\|^2 - n \e^2 +  n \e^2 / 2 = \|f\|^2 - n \e^2 / 2 < 0. \qedhere
	\end{align*}
\end{proof}

Now the question is: Which filters (more precisely, filtered \LL-semigroups) have the difference-Ramsey property? We give an answer to this based on a notion of differentiation for subsets of a semigroup that we define in \cref{sec:differentiation}. Our main theorem (\cref{st:difference-Ramsey_for_del-filters}) states that the filters that respect this notion of differentiation in an appropriate sense have the difference-Ramsey property. These filters include all of those for which the van der Corput property was known, so, as a corollary, we obtain a van der Corput lemma generalizing its previously known instances.

\begin{acknowledgements}
I am grateful to \Slawek Solecki for very useful suggestions and comments. Many thanks to John H. Johnson, Joel Moreira, Florian K. Richter, and Donald Robertson for their enlightening remarks, corrections, and references. I also thank James Cummings and Stevo \Todorcevic for their helpful comments and positive feedback.
\end{acknowledgements}

\section{Preliminaries}\label{sec:prelim}

\subsection{Filters}\label{subsec:filters}

\begin{defn}
A filter $\FC$ on a set $S$ is a nonempty collection of subsets of $S$ that does not contain $\0$ and is closed upward\footnote{A family $\CC \subseteq \Pow(X)$ is \emph{upward} (resp. \emph{downward}) \emph{closed} if $A \in \CC$ implies $B \in \CC$ for every superset (resp. subset) $B \subseteq X$ of $A$.} and under finite intersections.
\end{defn}

Note that $\dualF := \set{A \subseteq S : A^c \in \FC}$ is an ideal and we call it the \emph{dual ideal of $\FC$}. Thus, $\AC_\FC := \dualF \cup \FC$ is an algebra with a $\set{0,1}$-valued finitely additive complete measure $\mu_\FC$ defined on it such that the measure-$1$ sets are exactly those in $\FC$.

We call a set $A \subseteq S$ 
\begin{itemize}
\setlength{\itemsep}{4pt}
\item \emph{$\FC$-large} if $A \in \FC$ (i.e. $A$ has measure $1$);
\item \emph{$\FC$-small} if $A^c$ is $\FC$-large (i.e. $A$ has measure $0$);
\item \emph{$\FC$-positive}, and write $A >_\FC 0$, if $A$ is not $\FC$-small (i.e. either $A$ has measure $1$ or the measure of $A$ is undefined).
\end{itemize}

We denote the collection of $\FC$-positive sets by $\Pos$ and it is often helpful to think of them as nonempty open sets.

For sets $A,B \subseteq S$, we write $A \simF B$ if $A \simdiff B$ is $\FC$-small. This clearly defines an equivalence relation, so we say that $\CC \subseteq \Pow(S)$ is $\FC$-invariant if for sets $A,B \subseteq S$ with $A \simF B$, $A \in \CC$ implies $B \in \CC$. We write $A \asubset B$ if $A \setminus B$ is $\FC$-small; equivalently, $A \cap H \subseteq B$ for some $\FC$-large $H \subseteq S$.

For a set $A \subseteq S$ and a property $P(\cdot)$ of elements of $S$, we write 
$$
\forall^\FC s \in A \ P(s)
$$ 
to mean that for all but an $\FC$-small set of $s$ in $A$, $P(s)$ holds; consequently, we write 
$$
\exists^\FC s \in A \ P(s)
$$ 
to mean $\neg \forall^\FC s \in A \ \neg P(s)$, i.e. there exists an $\FC$-positive set of $s$ in $A$ (in particular $A$ is $\FC$-positive) such that $P(s)$ holds.

Lastly, we recall the notion of a limit along a filter. For a topological space $X$, a sequence $(x_s)_{s \in S} \subseteq X$ and a point $x \in X$, we write
$$
\lim_{s \to \FC} x_s = x
$$
if for every open neighborhood $U \subseteq X$ of $x$, we have $\forall^\FC s \in S \; (x_s \in U)$.

\subsection{Examples of filters on semigroups}

In the sequel, we consider filters on semigroups and we start by listing some examples. Henceforth, let $G$ denote a semigroup.

\subsubsection{Almost invariant filters}

A filter $\FC$ on $G$ is called \emph{invariant} if for every $A \subseteq G$,
$$
A \text{ is $\FC$-large } \imp \forall g \ (A g^{-1} \text{ is $\FC$-large}),
$$
where $A g^{-1} := \set{h \in G : hg \in A}$. Thinking of $\FC$ as a finitely additive measure, this simply means that it is invariant under the right translation action of $G$ on itself.

A filter $\FC$ on $G$ is called \emph{almost invariant} if for every $A \subseteq G$,
$$
A \text{ is $\FC$-large } \imp \forall^\FC g \ (A g^{-1} \text{ is $\FC$-large}).
$$
Working with almost invariant filters, it is convenient to use the following notation: for a set $A \subseteq G$, put
$$
\St_\FC(A) := \set{g \in G : A g^{-1} \text{ is $\FC$-large}}.
$$
Thus, for an almost invariant filter $\FC$, if $A$ is $\FC$-large then so is $\St_\FC(A)$.
    
The class of almost invariant filters includes many important examples, most of which are actually invariant.

\begin{examples}\label{examples:filters}
\item\label{example:Frechet_filter} The \emph{\Frechet filter} on a group $G$, i.e. the filter containing all cofinite subsets of $G$, is invariant.
    
\item The filter $\FC_\mu$ of \emph{conull sets} of a finitely additive, or even subadditive, invariant nonzero measure $\mu$ on a semigroup $G$ is invariant.

\item The density filter (along a fixed \Folner sequence) $\FC_d$ on an amenable group $G$ is invariant.

\item If $G$ is a Polish group (or more generally a Baire group\footnote{A topological group $G$ is called Baire if it is not meager, i.e. a countable union of nowhere dense sets. Examples are Polish groups, as well as locally compact Hausdorff groups.}), then \emph{the collection of comeager sets} forms a filter; in fact, this filter is closed under countable intersections. Due to continuity of group multiplication, this filter is invariant under right multiplication.

\item For $G = \N$, call $A \subseteq \N$ a set of convergence if $\sum_{n \in A} {1 \over n} < \w$. Clearly, sets of convergence form an ideal and hence the collection $\Fh$ of complements of sets of convergence is a filter. Moreover, if $A$ is a set of convergence then so is $A - n$ for any $n \in \N$; thus, $\Fh$ is invariant.

\item Finally, \emph{idempotent ultrafilters}, i.e. maximal filters $p$ on a semigroup $G$ such that for any $A \subseteq G$,
$$
A \text{ is $p$-large } \shortiff \forall^p g \ (A g^{-1} \text{ is $p$-large}).
$$
In particular, idempotent\footnote{These ultrafilters are called idempotent because their defining condition is equivalent to $p \ast p = p$, where $\ast$ is the convolution operation defined in the same way as for measures.} ultrafilters are almost invariant. They always exist on any semigroup by Ellis's theorem, see, for example, \cite{Todorcevic}*{2.1 and 2.9}.
\end{examples}
	
\subsubsection{\IP-sets and the filter $\IPstar$}

For $n \in \N \cup \set{\w}$, put $\n := \set{0,1,...,n-1}$ if $n \in \N$ (although, set-theoretically there is not difference between $n$ and $\n$) and put $\n := \N$ if $n = \w$. Let $G$ be a semigroup. For a countable (or finite) sequence $(g_i)_{i < n}$, $n \in \N \cup \set{\w}$, of elements of $G$ and finite $\0 \ne \al \subseteq \n$, put $g_\al = g_{i_1} g_{i_2} ... g_{i_k}$, where $i_1 > i_2 > \cdots > i_k$ list the elements of $\al$ in the decreasing order; also put $g_\0 = 1_G$. Finally, let $\FP(g_i)_{i<n} = \set{g_\al : \al \subseteq \n \text{ finite}}$ and call it a \emph{finite product set} of length $n$.

A subset $A \subseteq G$ is called an \emph{IP-set} (stands for Infinite-dimensional Parallelepiped) if it is a finite product set of infinite length, i.e. $A = \FP(g_n)_{n \in \N}$ for some sequence $(g_n)_{n \in \N}$ of (not necessarily distinct) elements of $G$. There is a tight connection between $\IP$-sets and idempotent ultrafilters. Firstly, by \cite{Berg_IP}*{Theorem 2.5}\footnote{Although \cite{Berg_IP}*{Theorem 2.5} is stated and proved for $G=\N$, the same proof works for any semigroup.}, we have that every IP-set $A \subseteq G$ supports an idempotent ultrafilter, i.e. there is an idempotent ultrafilter $p$ on $G$ so that $A$ is $p$-large. Conversely, we have the following standard fact:

\begin{prop}
	For any semigroup $G$ and an idempotent ultrafilter $p$ on $G$, every $p$-large set contains an IP-set.
\end{prop}
\begin{proof}
	Let $A \subseteq$ be $p$-large. We recursively define sequences $(g_n)_{n \in \N}$ of elements of $G$ and $(A_n)_{n \in \N}$ of $p$-large subsets of $G$ such that
	\begin{enumerate}[(i)]
		\item $A_0 = A$,
		\item $A_{n+1} = A_n \cap A_n g_n^{-1}$,
		\item $g_n \in A_n$,
	\end{enumerate}
	and we do it as follows: having $A_n$ defined and $p$-large, by almost invariance, we know that $\St_p(A_n)$ is also $p$-large, so in particular $A_n \cap \St_p(A_n) \ne \0$ and we take $g_n \in A_n \cap \St_p(A_n)$. Thus, $A_n g_n^{-1}$ is $p$-large, and hence such is $A_{n+1} := A_n \cap A_n g_n^{-1}$, finishing the construction. Now it is easy to check that $\FP(g_n)_{n \in \N} \subseteq A$.
\end{proof}

Thus, we get:

\begin{cor}\label{st:contains_IP_iff_supports_idempotent_ultrafilter}
	A set $A \subseteq G$ contains an IP-set if and only if it supports an idempotent ultrafilter.
\end{cor}

This corollary in its turn implies the following famous theorem (see \cite{Furst_book}*{Proposition 8.13}):

\begin{theorem}[Hindman]
The class of IP-sets is Ramsey, i.e. if an IP-set is partitioned into finitely many subsets, one of these subsets contains an IP-set.
\end{theorem}

This theorem allows us to define a filter $\IPstar$ for which the positive sets are exactly those that contain an IP-set:
$$
\IPstar = \set{F \subseteq G : F \text{ meets every IP-set}}.
$$
To see that this is indeed closed under finite intersections, first note the following:
\begin{lemma}\label{lemma_F_intersect_IP_contains_IP}
For every $\IPstar$-large $F$ and IP-set $A \subseteq G$, $F \cap A$ contains an IP-set.
\end{lemma}
\begin{proof}
Immediately follows from Hindman's theorem and the definition of $\IPstar$.
\end{proof}

We can now easily conclude:
\begin{prop}
$\IPstar$ is a filter.
\end{prop}
\begin{proof}
Let $F_1$ and $F_2$ be $\IPstar$-large and we need to show that so is $F_1 \cap F_2$. To this end, fix an IP-set $A \subseteq G$. By Lemma \ref{lemma_F_intersect_IP_contains_IP}, there is an IP-set $A' \subseteq F_1 \cap A$. But by the same lemma, there is a further IP-set $A'' \subseteq F_2 \cap A'$, and thus $A'' \subseteq (F_1 \cap F_2) \cap A$, so in particular, $(F_1 \cap F_2) \cap A \ne \0$.
\end{proof}

Moreover, \cref{st:contains_IP_iff_supports_idempotent_ultrafilter} implies that $\IPstar$ is the intersection of all idempotent ultrafilters on $G$.

\begin{remark*}
Because of the latter fact, the statements below that are true for $\IPstar$ can be derived indirectly from them being true for idempotent ultrafilters.
\end{remark*}

\subsection{\M-semigroups}\label{subsec:M-semigroups}

In literature, a measurable space is a pair $(S,\AC)$, where $X$ is a set and $\AC$ is a $\si$-algebra on $S$. For the purpose of stating the van der Corput property, we would like to drop the last requirement, and for an arbitrary collection $\PC \subseteq \Pow(S)$ of subsets of $S$, we call the pair $(S, \PC)$ an \emph{\M-space}\footnote{Here \M stands for measurable, but we use the term ``\M-space'' rather than ``measurable space'' to avoid abuse of terminology and confusion.} (also an \emph{\M-semigroup} if $S$ is a semigroup). We refer to the sets in $\PC$ as \emph{$\PC$-measurable sets} and we extend this over functions in the following definition.

\begin{defn}\label{defn:M-measurability}
	For an \M-space $(S, \PC)$, a function $f : S \to \R$ is said to be \emph{upper $\PC$-semimeasurable} if for each $r \in \R$, $f^{-1}([r,\w)) \in \PC$. For a Hilbert space $\Hil$ with inner product $\gen{\cdot, \cdot}$, a function (i.e. sequence) $e : S \to \Hil$ is said to be \emph{weakly upper $\PC$-semimeasurable} if for every $h \in \Hil$, the function $S \to \displaystyle \R^+$, given by $s \mapsto |\gen{h, e(s)}|$, is upper $\PC$-semimeasurable.
\end{defn}

\begin{remark}\label{remark:equivalence_of_measurability_for_si-algebras}
It is not hard to check that if $\PC$ is actually a $\si$-algebra, then upper $\PC$-semimeasurability for a function $f : S \to \R$ coincides with the classical notion of $\PC$-measurability. Consequently, for a function $e : S \to \Hil$, being weakly upper $\PC$-semimeasurable is the same as being weakly $\PC$-measurable in the classical sense, i.e. as a function from the measurable space $(S, \PC)$ to the measurable space $(\Hil, \BC_\mathrm{w}(\Hil))$, where $\BC_\mathrm{w}(\Hil)$ is the Borel $\si$-algebra of $\Hil$ with respect to the weak topology.
\end{remark}

Finally, we remark that the van der Corput property is really a property of a filter on an \M-semigroup and to make it easy to state we fix the following terminology: for a filter $\FC$ on a set $S$ and $\PC \subseteq \Pow(S)$, we refer to the triple $(S,\PC,\FC)$ as a \emph{filtered \M-space} (also a \emph{filtered \M-semigroup} if $S$ is a semigroup). Note, however, that we do not impose any relation between $\PC$ and $\FC$.

\subsection{Notions of largeness and filtered \LL-semigroups}

We now generalize the notion of a filter by dropping the closure under intersections requirement.

\begin{defn}
	For a nonempty set $S$, call a collection $\LC \subseteq \Pow(S)$ a \emph{largeness notion on $S$} if it is nonempty, upward closed and $\0 \notin \LC$. We refer to the pair $(S,\LC)$ as an \emph{\LL-space} (also an \emph{\LL-semigroup} if $S$ is a semigroup).
\end{defn}

Throughout, we will use the following terminology. For a largeness notion $\LC \subseteq \Pow(S)$, and we say a set $A \subseteq S$ is \emph{$\LC$-large} to mean $A \in \LC$. Furthermore, for a property $P$ of points $s \in S$, we write
\begin{center}
	there are $\LC$-many $s \in S$ for which $P(s)$ holds
\end{center}
to mean that the set $\set{s \in S : P(s) \text{ holds}}$ is $\LC$-large.

Clearly, any filter $\FC$ on $S$ is a largeness notion, but $\Pos$ is also a largeness notion, while it may not be a filter. For our purposes, we will consider semigroups with a fixed largeness notion and a filter, so we make the following definition.

\begin{defn}\label{defn:filtered_L-space}
	Let $(S,\LC)$ be an \LL-space. A filter $\FC$ on $S$ is said to be compatible with $\LC$ if $\FC \subseteq \LC$. We refer to such triple $(S, \LC, \FC)$ as a \emph{filtered \LL-space} (also a \emph{filtered \LL-semigroup} if $S$ is a semigroup).
\end{defn}


\section{Differentiability of subsets of semigroups}\label{sec:differentiation}

We now proceed to define a stratification of a given largeness notion on a semigroup $G$; namely, we introduce quantitative strengthening of largeness using the semigroup operation. This is analogous to $n$-differentiability being a quantitative strengthening of continuity for functions on $\R$, thinking of continuity as $0$-differentiability.

Throughout this section, let $\LC$ and $\LC'$ denote largeness notions on a semigroup $G$.

\subsection{Definition and examples}

We set up notation that emphasizes the analogy with differentiation.

\begin{notation}
For $A \subseteq G$ and $g \in G$, put 
$$
\del_g A := A \cap A g^{-1}
$$
and call it \emph{the directional derivative of $A$ in the direction of $g$}.
\end{notation}

\begin{defn}\label{defn:differentiation}
For $A \subseteq G$, put $\De^0(A / \LC, \LC') := A$ and call $A$ \emph{$0$-differentiable over $(\LC,\LC')$} if $A \in \LC$. For $n \ge 1$, we recursively define sets $\De^n(A / \LC, \LC')$ and notion of $n$-differentiability as follows: let $\De^n(A / \LC, \LC')$ denote the set of all directions $g \in G$ in which the derivative of $A$ is $(n-1)$-differentiable, that is,
$$
\De^n(A / \LC, \LC') := \set{g \in G : \del_g A \text{ is $(n-1)$-differentiable over $(\LC, \LC')$}},
$$
and call $A$ \emph{$n$-differentiable over $(\LC, \LC')$} if $\De^n(A / \LC, \LC') \in \LC'$. We say that $A \subseteq G$ is \emph{$\w$-differentiable over $(\LC, \LC')$} if it is $n$-differentiable over $(\LC,\LC')$ for every $n \in \N$. For $n \le \w$, let $C^n(\LC, \LC') \subseteq \LC$ denote the collection of all $n$-differentiable sets over $(\LC,\LC')$.
\end{defn}

Before illustrating the notion of differentiability on some examples, we record the following self-induction phenomenon, which, in the author's opinion, is what often lies at the heart of statements of the form
\begin{center}
	one recurrence $\implies$ multiple recurrence,
\end{center}
mentioned in the introduction.

\begin{lemma}\label{st:self-induction}
	If $C^1(\LC, \LC') = \LC$, then $C^\w(\LC, \LC') = \LC$.
\end{lemma}
\begin{proof}
	Every set in $\LC$ being $1$-differentiable over $(\LC, \LC')$, means, by definition, that $\De^2(A / \LC, \LC') = \De^1(A / \LC, \LC')$ for every $A \in \LC$. Iterating this gives $\De^n(A / \LC, \LC') = \De^1(A / \LC, \LC')$ for every $A \in \LC$, which implies $C^\w(\LC, \LC') = \LC$.
\end{proof}

We also record that differentiability over $(\LC, \LC')$ is upward closed with respect to $(\LC, \LC')$.

\begin{prop}\label{st:diff_is_hereditary-in-L}
	For $i = 0,1$, let $\LC_i,\LC_i'$ be largeness notions on $G$ with $\LC_i \subseteq \LC_i'$. Then for any $n \le \w$, $C^n(\LC_0, \LC_1) \subseteq C^n(\LC_0', \LC_1')$.
\end{prop}
\begin{proof}
	Straightforward induction on $n$.	
\end{proof}

Although \cref{defn:differentiation} is stated for potentially different $\LC$ and $\LC'$, it will primarily be used with $\LC = \LC'$. Thus, we make the following convention.

\begin{convention}
	For $\LC = \LC'$, we say \emph{$n$-differentiable over $\LC$}, instead of $n$-differentiable over $(\LC,\LC)$, and write $\De^n(A / \LC), C^n(\LC)$, in lieu of $\De^n(A / \LC, \LC), C^n(\LC, \LC)$. Furthermore, we omit writing the superscript $n$ in $\De^n$ if $n = 1$. Lastly, we colloquially refer to the sets of the form $\De(A / \LC)$, for some $A \subseteq G$, as \emph{$\De$-sets over $\LC$}, and we refer to the elements in $\De(A / \LC)$ as \emph{differentiation-friendly directions for $A$}.
\end{convention}

\begin{examples}\label{examples_differentiable_sets}
\item\label{example_differentiability_for_invariant_filter} For an almost invariant filter $\FC$ on $G$, \emph{every $\FC$-large set is $\w$-differentiable over $\FC$.} To show this, by \cref{st:self-induction}, we only need to verify that for every $\FC$-large set $A$, the set $\De(A / \FC)$ is still $\FC$-large, which follows from the fact that it contains $\St_\FC(A)$.

\item\label{example_differentiability_for_Frechet} Let $\FC$ be the \Frechet filter on $G$. This filter is not almost invariant in general, but even when it is (e.g. when $G$ is a group or $G = \N$, $\FC$ is invariant), there are still many $\FC$-positive subsets $A \subseteq G$ that are not even $1$-differentiable over $\Pos$. For instance, let $G = \N$ and take any $A \subseteq \N$ with superlinear growth rate, i.e. for any $k \in \N$, $A \cap \del_k A$ is finite, e.g. $A = \set{2^n : n \in \N}$. Then $A$ is $0$-differentiable over $\Pos$, but it is not $1$-differentiable over $\Pos$ because $\De(A / \Pos) = \0$.

\item\label{example_differentiability_for_IP_filter} For the filter $\IPstar$ on a semigroup $G$, \emph{every $\IPstar$-positive set is $\w$-differentiable over $\Pos[(\IPstar)]$}. To show this, by \cref{st:self-induction}, it is enough to verify that for every $\IPstar$-positive set $A \subseteq G$, the set $\De(A / \Pos[(\IPstar)])$ is $\IPstar$-positive. Recalling that a set is $\IPstar$-positive precisely when it contains an IP-set, it is enough to prove the following stronger statement.
\begin{lemma}\label{lemma_IP_A_is_in_Delta(A)}
For every IP-set $A \subseteq G$ and every $g \in A$, the set $\del_g A$ contains an IP-set, and hence, $A \subseteq \De(A / \Pos[(\IPstar)])$.
\end{lemma}
\begin{proof}
Let $A = \FP(g_i : i \in \N)$ and fix $g \in A$. Thus, $g = g_{i_{n-1}} g_{i_{n-2}} ... g_{i_0}$ for some indices $i_{n-1} > i_{n-2} > ... > i_0$. But then, we still have $A g^{-1} \supseteq \FP(g_i : i > i_{n-1})$, and hence also $\del_g A \supseteq \FP(g_i : i > i_{n-1})$. Therefore, $g \in \De(A / \Pos[(\IPstar)])$, so $A \subseteq \De(A / \Pos[(\IPstar)])$.
\end{proof}
\end{examples}

\subsection{Properties of $\del$ and $\De$}

Throughout this subsection, all of the notions of differentiability are over $\LC$ and we will omit writing it; in particular, we will write $\De(\cdot)$ in lieu of $\De(\cdot / \LC)$.

The following proposition exhibits the connection between differentiability and finite product sets.

\begin{prop}\label{st:differentiable_set_contains_FP-shifts}
	For any $n$-differentiable set $A \subseteq G$, there is a sequence $(h_i)_{i<n} \subseteq \De(A)$ such that $\bigcap_{\al \subseteq \n} A h_\al^{-1} = \del_{h_{n-1}} ... \del_{h_1} \del_{h_0} A \in \LC$. In particular, $A$ contains $\LC$-many shifts of a finite product set of length $n$; that is, for every $g \in \del_{h_{n-1}} ... \del_{h_1} \del_{h_0} A$, $g \cdot \FP(h_i : i<n) \subseteq A$.
\end{prop}
\begin{proof}
	Take $h_0 \in \De^n(A)$ and applying induction to $\del_{h_0} A$ get $(h_i)_{1 \le i < n} \subseteq \De(\del_{h_0} A) \subseteq \De(A /)$ such that $\del_{h_{n-1}} ... \del_{h_1} \del_{h_0} A \in \LC$.
\end{proof}

Next, we record some relations between $\del$ and $\De$.

\begin{prop}\label{prop_algebraic_properties_of_del_De}
	Let $A \subseteq G$, $g,h \in G$, and $n,m \ge 0$.
	\begin{enumerate}[\textup{(\alph{enumi})}]
		\item\label{item_superassociativity_of_del} $\del$ is superassociative: $\del_{hg} A \supseteq \del_h \del g A$.
		
		\item\label{item_subcommutativity_of_De_del} $\De$ and $\del$ subcommute: $\De^n(\del_g A) \subseteq \del_g \De^n(A)$.
		
		\item\label{item_subassociativity_of_De} $\De$ is subassociative: $\De^{n+m}(A) \subseteq \De^n(\De^m(A))$. In particular, if $A$ is $(n+m)$-differentiable, then $\De^m(A)$ is $n$-differentiable.
	\end{enumerate}
\end{prop}
\begin{proof}
	For (a) just compute: 
	$
	\del_h \del_g A = A \cap A g^{-1} \cap A h^{-1} \cap A g^{-1} h^{-1} \subseteq A \cap A g^{-1} h^{-1} = \del_{hg} A.
	$
	
	For (b), we assume that $n \ge 1$ since $n=0$ case follows from the convention that $\De^0(B) = B$ for any $B \subseteq G$. Note that $\del_g A \subseteq A$ trivially implies $\De^n(\del_g A) \subseteq \De^n(A)$. As for $\De^n(\del_g A) \subseteq \De^n(A) g^{-1}$, fixing $h \in \De^n(\del_g A)$, we have that $\del_h \del_g A$ is $(n-1)$-differentiable and is contained in $\del_{hg} A$, by part (a). So, $\del_{hg} A$ is $(n-1)$-differentiable as well, and hence, $hg \in \De^n(A)$.
	
	We prove (c) by induction on $n$. The $n=0$ case follows from the convention, so we suppose 
	that the statement is true for $n \ge 0$ and prove for $n+1$. Fixing $g \in \De^{n+m+1}(A)$, we have that $\del_g A$ is $(n+m)$-differentiable, and hence, by induction, $\De^m(\del_g A)$ is $n$-differentiable. But by part (b),
	$
	\De^m(\del_g A) \subseteq \del_g \De^m(A),
	$
	so $\del_g \De^m(A)$ is also $n$-differentiable, and thus, $g \in \De^{n+1}(\De^m(A))$.
\end{proof}

Taking $m=1$ and replacing $n$ with $n-1$ in part (c) of \cref{prop_algebraic_properties_of_del_De}, we get:

\begin{cor}\label{cor_De-set_is_differentiable}
	For any $A \subseteq G$ and $n \ge 1$, 
	$$
	\De^{n-1}(\De(A)) \supseteq \De^n(A).
	$$
	In particular, if $A$ is $n$-differentiable, then $\De(A)$ is $(n-1)$-differentiable.
\end{cor}

Thus, containing a $\De$-set of an $n$-differentiable set is a structurally strong way of being $(n-1)$-differentiable.

We now arrive at a property that illustrates that the connection between $A$ and $\De(A)$ is tighter than it appears on the surgace of the definition of $\De(A)$. We refer to this property as \emph{Main Property of $\De$-sets}.

\begin{namedthm}{Main Property of $\De$-sets}\label{namedthm:main_property_of_De-sets}
	For $n \ge 1$ and $n$-differentiable $A \subseteq G$, there are $\LC$-many $g \in \De(A)$ (namely, all $g \in \De^n(A)$) such that
	\begin{goodenum}[(\roman{enumi})]
		\item $\del_g A$ is $(n-1)$-differentiable;
		
		\item $\del_g \De(A) \supseteq \De(\del_g A)$; in particular, $\del_g \De(A)$ is $(n-2)$-differentiable, if $n \ge 2$.
	\end{goodenum}
\end{namedthm}
\begin{proof}
	Because $A$ is $n$-differentiable, we have that $\De(A) \supseteq \De^n(A) \in \LC$. But for any $g \in \De^n(A)$, $\del_g A$ is $(n-1)$-differentiable. Moreover, by (b) of \cref{prop_algebraic_properties_of_del_De}, we have $\del_g \De(A) \supseteq \De(\del_g A)$.
\end{proof}

What this property says is that there are $\LC$-many directions $g$ in $\De(A)$ that are differenti-ation-friendly for both $A$ and $\De(A)$, simultaneously. Moreover, the property that $\De(A)$ contains all differentiation-friendly directions for $A$ is maintained by their directional derivatives along $g$, i.e. $\del_g \De(A)$ still contains all differentiation-friendly directions for $\del_g A$. We emphasize it here as it will play an important role later in defining a class of filters that ``respect'' all notions involved in the definition of differentiability.

\subsection{Derivation trees}

We now digress a bit to discuss the geometry underlying the definitions of differentiability and $\De$-sets. Although, this subsection provides an intuitive picture to keep in mind, nothing in it will be used in the proofs below, so it may be safely skipped.

We start with recalling some terminology regarding set-theoretic trees.

\begin{notation}
	For a set $X$, we denote by $\Fin{X}$ the set of finite tuples of elements of $X$, i.e.
	$
	\Fin{X} := \bigcup_{n \in \N} X^n,
	$
	where $X^0 = \set{\0}$. For $s \in \Fin{X}$, we denote by $|s|$ the length of $s$; thus, $s$ is a function from $\set{0,1,...,|s|-1}$ to $X$. Recalling that functions are sets of pairs, the notation $s \subseteq t$ for $s,t \in \Fin{X}$ means that $|s| \le |t|$ and $s(i) = t(i)$ for all $i < |s|$. Finally, for $s \in \Fin{X}$ and $x \in X$, we write $s \conc x$ to denote the extension of $s$ to a tuple of length $|s|+1$ that takes the value $x$ at index $|s|$; we define $x \conc s$ analogously. 
\end{notation}

\begin{defn}
	For a set $X$, a subset $T$ of $\Fin{X}$ is called a (set theoretic) \emph{tree} on $X$ if it is closed downward under $\subseteq$, i.e. for all $s,t \in \Fin{X}$, if $t \in T$ and $s \subseteq t$, then $s \in T$.
\end{defn}

\begin{notation}
	For a tree $T$ on a set $X$ and $s \in T$, define the set of extensions of $s$ in $T$ by $\ext_T(s) := \set{x \in X : s \conc x \in T}$. Call $s$ a \emph{leaf} of $T$ if $\ext_T(s) = \0$, and call $T$ \emph{pruned} if it has no leaves. Furthermore, define
	$$
	\Depth(T) :=
	\left\{\begin{array}{ll}
	-1 & \text{if } T = \0 \\
	\max_{s \in T} |s| & \text{if $T \ne \0$ and this $\max$ exists} \\
	\w & \text{otherwise}
	\end{array}\right.,
	$$
	as well as
	$$
	\depth(T) :=
	\left\{\begin{array}{ll}
	-1 & \text{if } T = \0 \\
	\min \set{|s| : s \text{ is a leaf in } T} & \text{if $T$ has a leaf} \\
	\w & \text{otherwise}
	\end{array}\right..
	$$
\end{notation}

We call $\al \in X^\N$ an \emph{infinite branch} through $T$ if for each $n \in \N$, $\al \rest{n} \in T$. Note that a nonempty pruned tree $T$ (i.e. $\depth(T) = \w$) has an infinite branch, while in general, $\Depth(T) = \w$ does not imply this.

\begin{defn}
	For a largeness notion $\LC$ on $X$, call a tree $T$ on $X$ an \emph{$\LC$-tree} if for every $s \in T$, either $s$ is a leaf, or else, $\ext_T(s) \in \LC$.
\end{defn}

Now let $X = G$ be a semigroup. For $s = (h_0,h_1,...,h_{m-1}) \in \Fin{G}$ and $A \subseteq G$, put 
$$
\del_s A := \del_{h_{m-1}} ... \del_{h_1} \del_{h_0} A
$$ 
with convention that $\del_\0 A = A$. Fixing largeness notions $\LC, \LC' \subseteq \Pow(G)$, we associate a tree to every subset of $G$.
 
\begin{defn}
	For a set $A \subseteq G$, the \emph{derivation tree of $A$ over $(\LC, \LC')$}, noted $\T{A}$, is the tree on $G$ defined by
	$$
	s \in \T{A} \iff \forall t \subseteq s \ \del_t A \text{ is $(|s|-|t|)$-differentiable over $(\LC, \LC')$}.
	$$
\end{defn}

Note that $\T{A}$ is actually an $\LC'$-tree on $\De(A / \LC, \LC')$; in fact, for each $s \in \T{A}$ that is not a leaf, $\ext_{\T{A}}(s) = \De(\del_s A / \LC, \LC')$.

\begin{prop}\label{prop:tree-properties}
	Let $A \subseteq G$.
	\begin{goodenum}
		\item $\T{A} = \0$ if and only if $A \notin \LC$.
		
		\item $\Depth(\T{A}) = \sup\set{n \in \N : A \text{ is $n$-differentiable over $(\LC, \LC')$}}$.
		
		\item\label{item:self-induction=>w-branch} If $C^1(\LC, \LC') = \LC$, then for any $A \in \LC$, $\T{A}$ is nonempty pruned; in particular, it has an infinite branch.
		
		\item\label{item:infinite-branch=>IP-set} If $\T{A}$ has an infinite branch, then $\De(A / \LC, \LC')$ contains an \IP-set. In fact, if $(h_n)_{n \in \N}$ is an infinite branch of $\T{A}$, then $\De(A / \LC, \LC') \supseteq \FP(h_n)_{n \in \N}$.
	\end{goodenum}
\end{prop}
\begin{proof}
	Follows from the definitions.
\end{proof}

For the rest of this subsection, we take $\LC = \LC'$, and use the term \emph{derivation tree of $A$ over $\LC$} instead of derivation tree of $A$ over $(\LC, \LC)$, and write $\T[\LC]{A}$ in lieu of $\T[\LC, \LC]{A}$.

\begin{examples}
	\item For an almost invariant filter $\FC$ on $G$, taking $\LC = \FC$, part \ref{item:self-induction=>w-branch} above and \examplesref{examples_differentiable_sets}{example_differentiability_for_invariant_filter} imply that for any $\FC$-large set $H \subseteq G$, $\T[\FC]{H}$ is nonempty pruned. In particular, $\T[\FC]{H}$ has an infinite branch, so, by \ref{item:infinite-branch=>IP-set} above, $\De(H / \FC)$ contains an \IP-set.
	
	\item Similarly, for the filter $\IPstar$ on $G$, taking $\LC = \Pos[(\IPstar)]$, it follows from \examplesref{examples_differentiable_sets}{example_differentiability_for_IP_filter} that for any $\IPstar$-positive set $A \subseteq G$, $\T[\LC]{A}$ is nonempty pruned, so $\De(A / \LC)$ contains an \IP-set. However, the latter is not news as \emph{any} $\IPstar$-positive set contains an \IP-set.
\end{examples}

Iterating \namedthmref{Main Property of $\De$-sets}{namedthm:main_property_of_De-sets}, we get that the derivation trees of $A$ and $\De(A / \LC)$ over $\LC$ have much in common.

\begin{prop}\label{prop_main_property_of_De-sets_with_derivation_trees}
	For $A \subseteq G$, the set $\T[\LC]{A} \cap \T[\LC]{\De(A / \LC)}$ contains an $\LC$-tree $T$ with $\depth(T) = \Depth(\T[\LC]{A})-1$ such that for each $s \in T$, $\del_s \De(A / \LC) \supseteq \De(\del_s A / \LC)$.
\end{prop}


\section{Intersections and differentiability}

Throughout this section, we fix an \LL-semigroup $(G,\LC)$ and all notions of differentiability will be over $\LC$. Here, we investigate the behavior of the notions associated with differentiation under intersections. More specifically, for sets $A,D \subseteq G$ with $A$ being $n$-differentiable over $\LC$, we are interested in the following two informal questions:
\begin{enumerate}[(1)]
	\item How much of $A$ does $D$ have to contain so that $A \cap D$ is still $n$-differentiable?
	
	\item How much of $\De(A)$ does $D$ have to contain so that $\De(A) \cap D$ still satisfies the Main Property of $\De$-sets (see \namedthmlabelref{namedthm:main_property_of_De-sets})?
\end{enumerate}
We give answers to both of these questions in the following two subsections.

\subsection{\thickL{n}ness}

Given sets $A \cap D$, we unravel the recursive definition of $A \cap D$ being $n$-differentiable and arrive at the following definition.

\begin{defn}\label{defn:del-thickness}
	Let $A, D \subseteq G$ and $n \in \N$, and suppose that $A$ is $n$-differentiable. For $n \ge 1$, we say that $D$ is \emph{\thickL{n} in $A$} if there are $\LC$-many $g \in G$ such that
	\begin{enumerate}[(i)]
		\item $\del_g A$ is $(n-1)$-differentiable;
		
		\item $\del_g D$ is \thickL{n-1} in $\del_g A$.
	\end{enumerate}
	For $n=0$, we say that $D$ is \emph{\thickL{0} in $A$} if $A \cap D$ is $0$-differentiable. Finally, in case $A$ is not $n$-differentiable, we simply declare $D$ to be \thickL{n} in $A$.
\end{defn}

We now check that this indeed gives us an answer to question (1) above.

\begin{prop}\label{st:del-thickness=intersection-differentiable}
	For $n \ge 0$, $A, D \subseteq G$ with $A$ being $n$-differentiable, $A \cap D$ is $n$-differentiable if and only if $D$ is \thickL{n} in $A$.
\end{prop}
\begin{proof}
	We prove by induction on $n$. For $n=0$, it is trivial, so let $n \ge 1$ and, assuming the equivalence is true for $n-1$, prove for $n$.
	
	\smallskip
	
	\noindent $\imp$: Suppose $A \cap D$ is $n$-differentiable. Then $\De^n(A \cap D) \in \LC$ and for any $g \in \De^n(A \cap D)$, $\del_g(A \cap D)$ is $(n-1)$-differentiable. But $\del_g(A \cap D) = (\del_g A) \cap (\del_g D)$, so by induction, $\del_g D$ is \thickL{n-1} in $\del_g A$.
	
	\smallskip
	
	\noindent $\rimp$: Suppose $D$ is \thickL{n} in $A$. Then $\exists^\FC g \in G$ such that $\del_g A$ is $(n-1)$-differentiable and $\del_g D$ is \thickL{n-1} in $\del_g A$. Thus, by induction, the set $(\del_g A) \cap (\del_g D)$ is $(n-1)$-differentiable, but it is equal to $\del_g(A \cap D)$, so we are done.
\end{proof}

\subsection{\DthickL{n}ness}

Using the analogy with Definition \ref{defn:del-thickness}, we define what it means for a set $D \subseteq G$ to contain a thick enough part of $\De(A / \LC)$ so that $D$ plays the same role for $A$ as $\De(A)$ in \namedthmref{Main Property of $\De$-sets}{namedthm:main_property_of_De-sets}.

\begin{defn}\label{defn:Delta-thickness}
Let $A, D \subseteq G$ and $n \in \N$, and suppose that $A$ is $n$-differentiable. For $n \ge 1$, we say that $D$ is \emph{\DthickL{n} for $A$} if there are $\LC$-many $g \in D$ such that
\begin{enumerate}[(i)]
\item $\del_g A$ is $(n-1)$-differentiable;

\item $\del_g D$ is \DthickL{n-1} for $\del_g A$.
\end{enumerate}
For $n=0$, we always say that $D$ is \emph{\DthickL{0} for $A$}. Finally, in case $A$ is not $n$-differentiable, we simply declare $D$ to be \DthickL{n} for $A$.
\end{defn}

It follows by an easy induction on $n$ that \DthickL{n}ness is closed upward, i.e. if $D' \supseteq D$, $A' \supseteq A$, and $D$ is \DthickL{n} for $A$, then $D'$ is \DthickL{n} in $A'$.

\smallskip

We now check that \cref{cor_De-set_is_differentiable} still (essentially) holds for \DthickL{n} sets.

\begin{prop}\label{prop_double_witness_in_thick} 
Let $A,D \subseteq G$, $n \ge 1$. Suppose that $A$ is $n$-differentiable and $D$ is \DthickL{n} for $A$. Then $\De^n(A) \cap \De^{n-1}(D)$ is in $\LC$, and thus, $D$ is $(n-1)$-differentiable.
\end{prop}
\begin{proof}
We prove by induction on $n$. For $n=1$, $\De^0(D) = D$ and $D$ being \DthickL{1} for $A$ is equivalent to $\De^1(A) \cap D \in \LC$. Let $n \ge 2$ and suppose the statement is true for $n-1$. By \DthickL{n}ness, there are $\LC$-many $g \in D \cap \De^n(A)$ such that $\del_g D$ is \DthickL{n-1} for $\del_g A$. But $\del_g A$ is $(n-1)$-differentiable, so by induction, $\De^{n-1}(\del_g A) \cap \De^{n-2}(\del_g D) \in \LC$. In particular, $\del_g D$ is $(n-2)$-differentiable, so $g \in \De^{n-1}(D)$. Since there are $\LC$-many such $g$ in $\De^n(A)$, it follows that $\De^n(A) \cap \De^{n-1}(D) \in \LC$.
\end{proof}

\begin{examples}\label{examples_De-thick_sets}
\item\label{example_invariant_filter_thick} Let $\LC = \Pos$, for an almost invariant filter $\FC$ on $G$. Then \emph{any $\FC$-large $H$ is \DthickL{n} for $A$ for any $A \subseteq G$ and $n \ge 0$}. Indeed, if $n \ge 1$ and $A$ is $n$-differentiable, then $\De^n(A) \asubset H \cap \St(H)$ and hence $\exists^\FC g \in H \cap \St(H)$ such that $\del_g A$ is $(n-1)$-differentiable. Moreover, since $g \in \St(H)$, $\del_g H$ is $\FC$-large, so, by induction on $n$, it must be \DthickL{n-1} for $\del_g A$.

\item\label{example_IP_filter_thick} Let $\LC = \Pos[(\IPstar)]$. Then \emph{every $\IPstar$-positive set $A$ is \DthickL{n} for $A$ for any $n \ge 0$}. Indeed, being $\IPstar$-positive, $A$ contains an \IP-set $P$. By \cref{lemma_IP_A_is_in_Delta(A)}, $P \subseteq \De(P) \subseteq \De(A)$, so for any $g \in P$, $\del_g A$ is positive, and hence $(n-1)$-differentiable. By induction on $n$, $\del_g A$ is \DthickL{n-1} for $\del_g A$.
\end{examples}

\smallskip

Lastly, we record what being \DthickL{n}ness means in terms of derivation trees.

\begin{prop}\label{prop_n-thickness_with_derivation_trees}
For $n \ge 0$ and $A, D \subseteq G$ with $A$ being $n$-differentiable, $D$ is \DthickL{n} for $A$ if and only if the tree $\T[\LC]{A} \cap \T[\LC]{D}$ contains an $\LC$-tree $T$ on $D$ with $\depth(T) = n - 1$ and such that for each $s \in T$, the set $\del_s D \cap \De(\del_s A / \LC)$ is $\LC$-large.
\end{prop}
\begin{proof}
	We prove by induction on $n$. For $n = 0$, this equivalence is trivial since both sides vacuously hold. So let $n \ge 1$ and, assuming the equivalence is true for $n-1$, prove for $n$.
	
	\smallskip

	\noindent $\imp$: Suppose $D$ is \DthickL{n} for $A$, so in particular, it follows from \cref{prop_double_witness_in_thick} that $D \cap \De(A / \LC)$ is $\LC$-large. We build the desired tree $T$ as follows. Let $L$ be the set of all $g \in D$ such that $\del_g A$ is $(n-1)$-differentiable and $\del_g D$ is \DthickL{n-1} for $\del_g A$; by \DthickL{n}ness, $L \in \LC$. For each $g \in L$, by induction, there is an $\LC$-tree $T_g \subseteq \T[\LC]{\del_g A} \cap \T[\LC]{\del_g D}$ with $\depth(T_g) = n - 2$ and such that for each $s \in T_g$, the set $\del_s \del_g D \cap \De(\del_s \del_g A / \LC)$ is $\LC$-large. It is now straightforward to check that the tree 
	$$
	T := \set{\0} \cup \set{g \conc s : g \in L, s \in T_g}
	$$
	is as desired. We only note that to verify $T \subseteq \T[\LC]{D}$, it is enough to show that $L \subseteq \De^{n-1}(D / \LC)$, which follows from \cref{prop_double_witness_in_thick} as for each $g \in L$, $\del_g D$ is \DthickL{n-1} for $A$.
	
	\smallskip
	
	\noindent $\rimp$: For $n \ge 2$, using the induction hypothesis, it is straightforward to check that for every $g \in \ext_T(\0)$, the conditions (i)--(ii) of the definition of \DthickL{n}ness hold. For $n = 1$, $T = \set{\0}$, which gives that $D \cap \De(A / \LC) = \del_\0 D \cap \De(\del_\0 A / \LC)$ is $\LC$-large, but the latter is equivalent to $D$ being \DthickL{1} for $A$.
\end{proof}


\section{Filters and differentiability}

Henceforth, we will be dealing with a fixed filter $\FC$ on a semigroup $G$ and our largeness notion will be $\LC = \Pos$. In particular, we will only use the notion of differentiability over $\Pos$, so we introduce special terminology and notation for this.

\begin{terminology}\label{term:diff-mod-filter}
	For a filter $\FC$ on a semigroup $G$ and $n \le \w$, call a set $A \subseteq G$ \emph{$n$-differentiable $\mmod{\FC}$} if it is $n$-differentiable over $\Pos$. (Note that $n$-differentiability over $\FC$ implies $n$-differentiability $\mmod{\FC}$, but the converse may not be true.) We also write $\De^n_\FC(A)$, $C^n_\FC$, $T_\FC(A)$, \thick{n}, \Dthick{n}, in lieu of $\De^n(A / \Pos)$, $C^n(\Pos)$, $\T[\Pos]{A}$, \thickL[\Pos]{n}, \DthickL[\Pos]{n}, respectively. Furthermore, we may omit writing the subscript $\FC$ altogether if the filter $\FC$ is clear from the context and there is no danger of confusion.
\end{terminology}

\subsection{Respecting differentiability}

We now define a class of filters for which the notion of differentiability is well-defined for subsets of $G / \simF$, i.e. for each $n \ge 0$, the collection $C^n_\FC$ is $\FC$-invariant.

\begin{defn}
A filter $\FC$ on $G$ is said to \emph{respect differentiability} if for all $n \ge 1$ and subsets $A \simF \tilA$ of $G$, $A$ is $n$-differentiable $\mmod{\FC}$ if and only if $\tilA$ is $n$-differentiable $\mmod{\FC}$.
\end{defn}

\begin{example}\label{example:respecting_diff_for_IPstar}
\emph{The filter $\IPstar$ on any semigroup $G$ respects differentiability.} This is because, by \examplesref{examples_differentiable_sets}{example_differentiability_for_IP_filter}, every $\IPstar$-positive set is automatically $\w$-differentiable $\mmod{\IPstar}$, so if $A \sim B$, then 
$$
A \text{ is $n$-differentiable } \mmod{\IPstar} \shortiff A >_\IPstar 0 \shortiff B >_\IPstar 0 \shortiff B \text{ is $n$-differentiable } \mmod{\IPstar}.
$$
\end{example}

The next proposition unravels the above definition to make it easier to check.

\begin{prop}\label{st:equiv_to_respecting_differentiability}
A filter $\FC$ on $G$ respects differentiability if and only if for every $\FC$-large set $H \subseteq G$, for every $n \ge 1$ and every $n$-differentiable set $A \subseteq G$, $\exists^\FC g \in G$ such that $H g^{-1} \asupset A'$ for some $(n-1)$-differentiable set $A' \subseteq \del_g A$.
\end{prop}
\begin{proof}
For $\imp$, note that if $A$ is $n$-differentiable and $H$ is $\FC$-large, then $A \cap H$ is still $n$-differentiable, so $\De^n(A \cap H)$ is $\FC$-positive. Thus, for any $g \in \De^n(A \cap H)$, $A' := \del_g (A \cap H) \subseteq \del_g A$ is $(n-1)$-differentiable and $A' \subseteq H g^{-1}$.

For $\rimp$, we assume the right-hand side and prove by induction on $n$ that if $H$ is $\FC$-large and $A$ is $n$-differentiable, then $A \cap H$ is also $n$-differentiable. The base case $n = 0$ is trivial because if $A$ is $\FC$-positive then so is $A \cap H$. Now suppose it holds for $n-1$, and let $A$ be $n$-differentiable and $H$ be $\FC$-large. Then, $\exists^\FC g \in G$ such that $H g^{-1} \asupset A'$ for some $(n-1)$-differentiable set $A' \subseteq \del_g A$, so, in particular, $A' \asubset \del_g H$. Thus, $A' \asubset \del_g (A \cap H)$ and hence $A' \cap H' \subseteq \del_g (A \cap H)$ for some $\FC$-large $H'$. By induction, $A' \cap H'$ is still $(n-1)$-differentiable, and hence so is $\del_g (A \cap H)$. The fact that this holds for $\FC$-positively many $g \in G$ means that $A \cap H$ is $n$-differentiable.
\end{proof}

\begin{example}\label{example:respecting_diff_for_almost_invariant_filters}
It follows immediately from the last proposition that \emph{any almost invariant filter respects differentiability}. This is because for every $g \in \De^n(A) \cap \St(H)$, $A' := \del_g A$ is $(n-1)$-differentiable and $Hg^{-1}$ is $\FC$-large, so $Hg^{-1} \asupset B$ holds for every set $B \subseteq G$, and in particular, for $B = A'$.
\end{example}

Finally, using \cref{st:del-thickness=intersection-differentiable}, we express respecting differentiability in terms of \thick{n}ness:
\begin{cor}
	A filter $\FC$ on a semigroup $G$ respects differentiability if and only if every $\FC$-large $H \subseteq G$ is \thick{n} in $A$ for every $A \subseteq G$ and $n \ge 0$.
\end{cor}

\begin{remark}\label{st:respecting_diff_compatible_with_C-infty}
	For a filter $\FC$ on $G$, respecting differentiability implies in particular that every $\FC$-large set is $\w$-differentiable $\mmod{\FC}$, i.e. $\FC \subseteq C^\w_\FC$. In other words, $\FC$ is compatible with the largeness notion $C^\w_\FC$, turning $(G, C^\w_\FC, \FC)$ into a filtered \LL-semigroup.
\end{remark}

\subsection{Respecting $\De$-sets}

Again fix an ambient filter $\FC$ on $G$. Just like it was natural to require $\FC$ respect differentiability, it is also natural to have $\FC$ respect the role that $\De_\FC(A)$ plays for $A$, i.e. ensure that \namedthmref{Main Property of $\De$-sets}{namedthm:main_property_of_De-sets} is stable under $\FC$-small perturbations. This amounts to the following two conditions:
\begin{enumerate}[(i)]
	\item For every set $A \subseteq G$ and $D \subseteq G$ with $D \simF \De_\FC(A)$, $D$ is \Dthick{n} for $A$.
	
	\item For $D \subseteq G$ and $A,\tilA \subseteq G$ with $A \simF \tilA$, $D$ is \Dthick{n} for $A$ if and only if $D$ is \Dthick{n} for $\tilA$.
\end{enumerate}

Note that since being \Dthick{n} is closed under supersets, condition (i) amounts to $\De_\FC(A) \cap H$ being \Dthick{n} for $A$ for every $\FC$-large $H$, which is the same as $H$ itself being \Dthick{n} for $A$.

As for condition (ii), below, we will need it only for $\FC$-large sets and their derivatives, so instead of demanding it for all sets $D$, we define the following strengthening of \Dthick{n}ness and require all $\FC$-large sets to satisfy it.

\begin{defn}\label{defn:absolute-Delta-thickness}
	Let $A, D \subseteq G$ and $n \in \N$. For $n \ge 1$, we say that $D$ is \emph{absolutely \Dthick{n} for $A$} if for any $n$-differentiable ($\mmod{\FC}$) $\tilA \simF A$, $\exists^\FC g \in D$ such that
	\begin{enumerate}[(i)]
		\item $\del_g \tilA$ is $(n-1)$-differentiable;
		
		\item $\del_g D$ is absolutely \Dthick{n-1} for $\del_g \tilA$.
	\end{enumerate}
	For $n=0$, we always say that $D$ is \emph{absolutely \Dthick{0} for $A$}.
\end{defn}

As with \Dthick{n}ness, it follows by an easy induction on $n$ that being absolutely \Dthick{n} is closed upward, i.e. if $D' \supseteq D$, $A' \supseteq A$, and $D$ is absolutely \Dthick{n} for $A$, then $D'$ is absolutely \Dthick{n} for $A'$.

\begin{prop}\label{st:absolute_Dthickness_for_almost_invariant_filters}
	For an almost invariant filter $\FC$ on $G$, every $\FC$-large set is absolutely \Dthick{n} for $A$ for every $A \subseteq G$ and $n \ge 0$.
\end{prop}
\begin{proof}
	The same as in \examplesref{examples_De-thick_sets}{example_invariant_filter_thick}.
\end{proof}

\begin{lemma}
	For the filter $\IPstar$ on $G$, every $\IP$-positive set $A$ is absolutely \Dthick{n} for $A$ for every $n \ge 0$.
\end{lemma}
\begin{proof}
	We prove by induction on $n$, and since $n=0$ is automatic, we assume that it is true for $k < n$ and prove for $n$. Let $\tilA \simF A$, so $\tilA \cap A$ is also $\IPstar$-positive and hence contains an \IP-set $P$. By \cref{lemma_IP_A_is_in_Delta(A)}, $P \subseteq \De(P) \subseteq \De(\tilA)$, so for any $g \in P$, $\del_g P$ is positive, and hence $(n-1)$-differentiable. By induction on $n$, $\del_g P$ is absolutely \Dthick{n-1} for $\del_g P$, and hence, by upward closure, $\del_g A$ is absolutely \Dthick{n-1} for $\del_g \tilA$.
\end{proof}

\begin{cor}\label{st:absolute_Dthickness_for_IPstar}
	For the filter $\IPstar$ on $G$, every $\IPstar$-large set $H$ is absolutely \Dthick{n} for $A$ for every $A \subseteq G$ and $n \ge 0$.
\end{cor}
\begin{proof}
	The set $H \cap A$ is $\IPstar$-positive, so by the previous lemma, it is absolutely \Dthick{n} for $H \cap A$, and thus, by upward closure, $H$ is absolutely \Dthick{n} for $A$.
\end{proof}

\subsection{$\del$-filters}

Finally, we define the class of filters, which respect all of the notions involved in differentiation.

\begin{defn}
A filter $\FC$ on a semigroup $G$ is called a \emph{$\del$-filter} (read \emph{del-filter}) if
\begin{enumerate}[(i)]
	\item $\FC$ respects differentiability (equivalently, every $\FC$-large $H$ is \thick{n} in $A$, for every $A \subseteq G$ and $n \ge 0$);

	\item Every $\FC$-large $H$ is absolutely \Dthick{n} for $A$, for every $A \subseteq G$ and $n \ge 0$.
\end{enumerate}
\end{defn}

By \cref{example:respecting_diff_for_almost_invariant_filters,example:respecting_diff_for_IPstar} together with \cref{st:absolute_Dthickness_for_almost_invariant_filters,st:absolute_Dthickness_for_IPstar}, all almost invariant filters, as well as the filter $\IPstar$, are $\del$-filters.

\section{The main results}

Throughout this section, all of the notions of differentiability will be $\mmod{\FC}$, for a filter $\FC$ under consideration. 

\subsection{A difference-Ramsey theorem for $\del$-filters}

We are now ready to state and prove our main theorem. Recall that for a filter $\FC$ on $G$ that respects differentiability, the triple $(G, C^\w_\FC, \FC)$ is a filtered \LL-semigroup.

\begin{theorem}[Difference-Ramsey for $\del$-filters]\label{st:difference-Ramsey_for_del-filters}
For a $\del$-filter $\FC$ on $G$, the filtered \LL-semigroup $(G, C^\w_\FC, \FC)$ has the difference-Ramsey property, i.e. for any binary relation $E \subseteq G^2$, if
$
\forall^\FC h \forall^\FC g \ E(g,gh),
$
then for any $A \in C^\w_\FC$ and $n \in \N$, there is a sequence $(g_i)_{i \le n} \subseteq A$ with
$
\forall i<j \ E(g_j, g_i).
$
\end{theorem}
\begin{proof}
First let's introduce some notation: for $\al, \be \subseteq \m := \set{0,1,...,m-1}$, we write $\be < \al$ if $\max(\be) < \min(\al)$ or one of $\al,\be$ is $\0$. For $h,h' \in G$, put
$
E(\cdot h, \cdot h') := \set{g \in G : E(gh,gh')},
$
and also 
$
E(\cdot, \cdot h) := \set{g \in G : E(g, gh)},
$
so the set 
$$
H := \set{h \in G : E(\cdot, \cdot h) \text{ is $\FC$-large}}
$$
is $\FC$-large by the hypothesis.

To prove this theorem, we will construct sequences $(h_m)_{m < n}$ of elements in $H$ and $(A_m)_{m \le n}$ of subsets of $A$ such that for every $m \le n$, we have
\begin{enumerate}[($m$.1)]
\item $A_m \subseteq \del_{h_{m-1}} ... \del_{h_1} \del_{h_0} A$ is $(n-m)$-differentiable;
\item for all $\al,\be \subseteq \m$ with $\0 \ne \be < \al$, $A_m \subseteq E(\cdot h_\al, \cdot h_\al h_\be)$;
\item for all $\al \subseteq \m$, $h_\al \in H$;
\item $H_m := \del_{h_{m-1}} ... \del_{h_1} \del_{h_0} H$ is absolutely \Dthick{n-m} for $A_m$.
\end{enumerate}

Granted such a sequence, we define the desired sequence $(g_i)_{i \le n}$ as follows: by ($n$.1) $A_n \ne \0$, so take $g \in A_n$ and for each $i \le n$, put $g_i = g h_{\al_i}$, where $\al_i = \set{n-1,n-2,...,i}$. Because $g \in A_n \subseteq \bigcap_{\al \subseteq \m} A h_\al^{-1}$, $g_i \in A$ for each $i \le n$. Also, for $i < j \le n$, taking $\al := \al_j = \set{n-1,n-2,...,j}$ and $\be := \al_i \setminus \al_j = \set{j-1,j-2,...,i}$ in ($n$.2), we get $E(g h_\al, g h_\al h_\be)$ and thus $E(g_j, g_i)$.

\medskip

Now we show how to recursively define the sequences $(h_m)_{m < n}$ and $(A_m)_{m \le n}$. For $m=0$, put $A_0 = A$ so it is $n$-differentiable, $H_0 = H$ is absolutely \Dthick{n} for $A_0$ because $\FC$ is a $\del$-filter, and ($0$.2)--($0$.3) are vacuous.

Now suppose that for $m<n$, the sequences $(h_k)_{k < m}$ and $(A_k)_{k \le m}$ are defined and satisfy conditions ($m$.1)--($m$.4). By ($m$.4), $\exists^\FC h_m \in H_m$ such that $\del_{h_m} A_m$ is $(n-m-1)$-differentiable and $H_{m+1} = \del_{h_m} H_m$ is absolutely \Dthick{n-m-1} for $\del_{h_m} A_m$. Because $h_m \in H_m$, $h_m h_\al \in H$ for all $\al \subseteq \m$, so ($m+1$.3) holds. In particular, for all $\al \subseteq \m$, $E(\cdot, \cdot h_m h_\al)$ is $\FC$-large, so the set
$$
A_{m+1} := \del_{h_m} A_m \cap \bigcap_{\al \subseteq \m} E(\cdot, \cdot h_m h_\al)
$$
is still $(n-m-1)$-differentiable and $H_{m+1}$ is still absolutely \Dthick{n-m-1} for $A_{m+1}$. Thus, ($m+1$.1) and ($m+1$.4) are verified, and it remains to check ($m+1$.2). To this end, fix $\al,\be \subseteq \cl{m+1}$ with $\0 \ne \be < \al$. Note that we only need to check the case when $m \in \be$ or $m \in \al$. If $m \in \be$, then $\al = \0$ and $h_\be = h_m h_{\be'}$ for some $\be' \subseteq \m$. Thus, it follows by the very choice of $A_{m+1}$ that $A_{m+1} \subseteq E(\cdot, \cdot h_m h_{\be'})$. On the other hand, if $m \in \al$, then $h_\al = h_m h_{\al'}$ for some $\al' \subseteq \m$, so $A_{m+1} \subseteq \del_{h_m} A_m$ and ($m$.2) imply
$$
A_{m+1} \subseteq A_m h_m^{-1} \subseteq E(\cdot h_{\al'}, \cdot h_{\al'} h_\be) h_m^{-1} = E(\cdot h_m h_{\al'}, \cdot h_m h_{\al'} h_\be) = E(\cdot h_\al, \cdot h_\al h_\be).
$$
\end{proof}

\subsection{A van der Corput lemma for $\del$-filters}

To make it convenient to state the van der Corput property below, for a filter $\FC$ on a semigroup $G$, we put
$$
\MeasF := \dualF \uplus C^\w_\FC.
$$

The last theorem, together with \cref{st:diff-Ramsey=>vdC}, immediately gives:

\begin{cor}[van der Corput lemma for $\del$-filters]\label{st:vdC_for_del-filters}
	For a $\del$-filter $\FC$ on a semigroup $G$, the filtered \M-semigroup $(G, \MeasF, \FC)$ has the van der Corput property, i.e. for every weakly upper $\MeasF$-semimeasurable bounded sequence $(e_g)_{g \in G}$ in a Hilbert space $\Hil$, we have
	$$
	\lim_{h \to \FC} \lim_{g \to \FC} \gen{e_g, e_{gh}} = 0 \implies \lim_{g \to \FC} \gen{f,e_g} = 0, \ \forall f \in \Hil.
	$$
\end{cor}

Note that the van der Corput property is hereditary with respect to the measurability restriction given by $\PC$, i.e. if a filtered \M-semigroup $(G, \PC, \FC)$ has it and $\PC \supseteq \PC'$, then $(G, \PC', \FC)$ also has it. With this in mind, we now explicitly list some previously known concrete instances of \cref{st:vdC_for_del-filters}.

\begin{instances}\label{instances:vdC_for_del-filters}
	\item\label{instance:vdC_for_almost_invariant} \emph{For any almost invariant filter $\FC$ on a semigroup $G$, the filtered \M-semigroup $(G, \AC_\FC, \FC)$ has the van der Corput property.} (Recall that $\AC_\FC = \dualF \uplus \FC$.) Indeed, almost invariant filters are $\del$-filters, so by the above corollary, $(G, \MeasF, \FC)$ has the van der Corput property, and hence so does $(G, \AC_\FC, \FC)$ because $\AC_\FC \subseteq \MeasF$ by \examplesref{examples_differentiable_sets}{example_differentiability_for_invariant_filter}. \emph{Note that in case $G$ is a group or $G = \N$, this includes the \Frechet filter on $G$.}
	
	\item(Bergelson--McCutcheon \cite{Berg-McCutch}*{Theorem 2.3}) \emph{For any idempotent ultrafilter $p$ on a semigroup $G$, the filtered \M-semigroup $(G, \Pow(G), p)$ has the van der Corput property.} This is a special case of the previous example because for any ultrafilter $\FC$, $\AC_\FC = \Pow(G)$.
	
	\item(Furstenberg \cite{Furst_book}*{Lemma 9.24}) \emph{The filtered \M-semigroup $(G, \Pow(G), \IPstar)$ has the van der Corput property.} Indeed, $\IPstar$ is a $\del$-filter and, by \examplesref{examples_differentiable_sets}{example_differentiability_for_IP_filter}, $\MeasF[\IPstar] = \Pow(G)$.
\end{instances}

\begin{remark}\label{remark:generalizing_Bessel_for_Frechet}
	When $G$ is a group or $G = \N$, the \Frechet filter $\FC$ on $G$ is invariant, so it is included in \instanceref{instances:vdC_for_del-filters}{instance:vdC_for_almost_invariant}. This can be viewed as a generalization of \cref{st:Bessel} because every bounded sequence $(e_g)_{g \in G} \subseteq \Hil$ of pairwise orthogonal vectors is weakly upper $\AC_\FC$-semimeasurable. However, to verify this last fact, we use \cref{st:Bessel} itself, so this generalization is somewhat tautological.
\end{remark}

The next thing we will do is define a natural subclass of almost invariant filters (hence $\del$-filters), for which $\MeasF$ is rich (typically the powerset of $G$). What is somewhat remarkable is that the richness of $\MeasF$ will be proven using, again, the difference-Ramsey theorem for $\del$-filters.

\section{\D-measures}

In this section, we define a notion of measure that generalizes invariant finitely additive probability measures and idempotent ultrafilters, as well as subadditive notions such as the notion of upper density for subsets of amenable groups. We study the differentiability of sets of positive measure and obtain a van der Corput lemma for these measures from \cref{st:vdC_for_del-filters}.

\subsection{Definitions and examples}

\begin{defn}
Let $S$ be a set and $\AC \subseteq \Pow(S)$ be an algebra. A \emph{finitely subadditive probability measure} on $\AC$ is a function $\mu : \AC \to [0,1]$ such that
\begin{enumerate}[(i)]
\item $\mu(\0) = 0, \mu(S) = 1$;
\item for $A,B \in \AC$, $A \subseteq B$ implies $\mu(A) \le \mu(B)$;
\item for $A,B \in \AC$, $\mu(A \cup B) \le \mu(A) + \mu(B)$.
\end{enumerate}
\end{defn}

For $(S,\AC,\mu)$ as above, we fix the following terminology. Call a set $B \subseteq S$ \emph{$\mu$-null} if $B \subseteq A$ for some $A \in \AC$ with $\mu(A) = 0$; denote by $\SC_\mu$ the collection of $\mu$-null sets. Similarly, call a set $B \subseteq S$ \emph{$\mu$-positive} if $B \supseteq A$ for some $A \in \AC$ with $\mu(A) > 0$; let $\LC_\mu$ denote the collection of $\mu$-positive sets and put $\AC_\mu := \SC_\mu \uplus \LC_\mu$. Clearly, $\AC_\mu \supseteq \AC$ and one can think of it as a completion of $\AC$ to a collection that enjoys the following dichotomy: each set in it is either small or it contains a structurally large set.

Furthermore, we call a set $C \subseteq S$ \emph{$\mu$-conull} if its complement is $\mu$-null. Note that $\mu$-null sets form an ideal and hence $\mu$-conull sets form a filter, which we denote by $\FC_\mu$. Also note that $\FC_\mu$ is compatible with the largeness notion $\LC_\mu$, so $(S, \LC_\mu, \FC_\mu)$ is a filtered \LL-space.

\smallskip

We now isolate a relevant class of finitely subadditive measures on a semigroup $G$. Below, we call $\PC \subseteq \Pow(G)$ \emph{invariant} if
$
\PC G^{-1} := \set{A g^{-1} : A \in \PC, g \in G} \subseteq \PC.
$

\begin{defn}
Let $\AC \subseteq \Pow(G)$ be an invariant algebra and $\mu$ be a finitely subadditive measure on $\AC$. We say that $\mu$ is \emph{almost invariant} if
\begin{enumerate}[(i)]
\setcounter{enumi}{3}
\item for every $A \in \AC$, 
$$
\forall^{\FC_\mu} g \; \mu(A g^{-1}) = \mu(A).
$$
\end{enumerate}
Furthermore, we say that $\mu$ is \emph{additive on translates} if
\begin{enumerate}[(i)]
\setcounter{enumi}{4}
\item for every $A \in \AC$ and $g_0,g_1,...,g_{n-1} \in G$,
$$
\mu(\bigcup_{i<n} A g_i^{-1}) \ge \sum_{i<n} \mu(A g_i^{-1}) - \sum_{i < j < n} \mu(A g_i^{-1} \cap A g_j^{-1}).
$$
\end{enumerate}
If $\mu$ satisfies both (iv) and (v), we call it a \emph{\D-measure}\footnote{Here, ``D'' stands for density as \D-measures can be viewed as generalizations of upper density on amenable groups.} and refer to $(G,\AC,\mu)$ as a \emph{\D-measured semigroup}.
\end{defn}

Before proceeding with examples, we record the following simple observations.

\begin{prop}
	Let $\mu$ be a finitely subadditive probability measure on an invariant algebra $\AC \subseteq \Pow(G)$.
	
	\begin{enumerate}[\textup{(\alph{enumi})}]
		\item If $\mu$ is almost invariant, then so is $\FC_\mu$, and hence $\FC_\mu$ is a $\del$-filter.
		
		\item If $\mu$ is additive on translates, then it is genuinely additive on almost disjoint translates, that is: for every $A \in \AC$ and $g_0,g_1,...,g_{n-1} \in G$, if $A g_i^{-1} \cap A g_j^{-1}$ is $\mu$-null whenever $i \ne j$, then
			$$
			\mu(\bigcup_{i<n} A g_i^{-1}) = \sum_{i<n} \mu(A g_i^{-1}).
			$$
	\end{enumerate}
\end{prop}

\begin{examples}\label{examples:D-measures}
\item Idempotent ultrafilters, or more precisely, the $\set{0,1}$-measures associated to idempotent ultrafilters, are examples of finitely additive \D-measures on a semigroup $G$ with $\AC = \Pow(G)$.

\item More generally, for any almost invariant filter $\FC$, the associated $\set{0,1}$-measure on $\AC_\FC = \dualF \cup \FC$ is a \D-measure. This includes, in particular, the \Frechet filter on any group.

\item\label{example:comeager_in_Baire_group} A notable instance of the previous example is the filter $\CC$ of comeager sets on a Baire topological group. In this case, $\AC$ is actually a $\si$-algebra, being the disjoint union of the collections of meager and comeager sets.

\item For a countable amenable group $G$ with a \Folner sequence $(F_n)_{n \in \N}$, we define the \emph{upper density function along $(F_n)_{n \in \N}$} by letting
$$
\updens(A) = \limsup_{n \to \w} {|A \cap F_n| \over |F_n|},
$$
for all sets $A \subseteq G$. The function $\updens$ is a fully invariant (but not additive) \D-measure on the powerset $\AC = \Pow(G)$. The condition of additivity on translates holds simply because the subsequence that achieves the $\limsup$ for $A$ also achieves it for all of its translates. We refer to the associated filter $\FC_{\updens}$ as \emph{the density filter along $(F_n)_{n \in \N}$}.

\item\label{example:conull_in_amenable_group} Every locally compact Hausdorff amenable group $G$, by definition, admits a finitely additive invariant probability measure on the $\si$-algebra $\AC$ of Haar measurable sets. In particular, if $G$ is countable, then $\AC = \Pow(G)$.

\item\label{example:conull_in_compact_group} Every compact Hausdorff group $G$ admits a countably additive invariant probability measure $\mu$, namely, the normalized Haar measure defined on the Borel $\sigma$-algebra of $G$. Thus, $(G, \AC, \mu)$ is, in particular, a \D-measured group, where $\AC$ is the $\si$-algebra of $\mu$-measurable sets.

\item\label{example:conull_in_ultraproduct} Let $(G_n, \AC_n, \mu_n)_{n \in \N}$ be a sequence of groups with finitely additive invariant probability measures $\mu_n$, and let $G$ be their ultraproduct. Then $G$ admits a \emph{countably additive} invariant probability measure $\mu$, called the \emph{Loeb measure}, defined on the $\sigma$-algebra $\AC$ generated by the so-called internal sets.

\item\label{example:conull_in_probability_group} Generalizing the last two examples, every probability group (see Definition 1 in \cite{me_prob_groups}), by definition, admits a countably additive invariant probability measure defined some $\sigma$-algebra $\AC$.
\end{examples}

\subsection{\D-measures and differentiability}

\begin{prop}[Quantitative differentiability for \D-measures]
Let $(G,\AC,\mu)$ be a \D-measured semigroup. Then for every $\mu$-positive set $A \in \AC$, $\exists^{\FC_\mu} h \ \mu(\del_h A) \ge {\mu(A)^2 \over 3}$.
\end{prop}
\begin{proof}
Put $\FC = \FC_\mu$, $\e = {\mu(A)^2 \over 3}$, and assume for contradiction that $\forall^\FC h \in G$, $\mu(\del_h A) < \e$. Thus, for every such $h \in G$, we also have $\forall^\FC g \in G \ \mu((\del_h A)g^{-1}) < \e$, by the almost invariance of $\mu$. Because $(\del_h A)g^{-1} = A g^{-1} \cap A (gh)^{-1}$, we get
$$
\forall^\FC h \forall^\FC g \ \mu(A g^{-1} \cap A (gh)^{-1}) < \e.
$$
Hence, defining
$$
E(g_1,g_2) :\shortiff \mu(A g_1^{-1} \cap A g_2^{-1}) < \e,
$$
for $g_1,g_2 \in G$, we have
$
\forall^\FC h \forall^\FC g \ E(g, gh).
$
Because $\mu$ is almost invariant, the set 
$$
\St_\mu(A) := \set{g \in G : \mu(A g^{-1}) = \mu(A)}
$$ 
is $\FC$-large and hence $\w$-differentiable $\mmod{\FC}$ because $\FC$ is almost invariant (see \examplesref{examples_differentiable_sets}{example_differentiability_for_invariant_filter}). Choosing $n = \left\lceil {3 \over \mu(A)}\right\rceil$, apply the difference-Ramsey property (\cref{st:difference-Ramsey_for_del-filters}) to $E$ and $\St_\mu(A)$, and get a sequence $(g_i)_{i<n} \subseteq \St_\mu(A)$ such that $\mu(A g_i^{-1} \cap A g_j^{-1}) < \e$ for all $i<j$. But then, by additivity on translates, we have
\begin{align*}
\mu(\bigcup_{i<n} A g_i^{-1}) &\ge \sum_{i<n} \mu(A g_i^{-1}) - \sum_{i<j<n} \mu(A g_i^{-1} \cap A g_j^{-1}) \\
&> n \mu(A) - {1 \over 2}n(n-1) \e \\
&\ge {3 \over \mu(A)} \mu(A) - {1 \over 2}\left({3 \over \mu(A)} + 1\right){3 \over \mu(A)} {\mu(A)^2 \over 3} \\
&= 3 - {3 + \mu(A) \over 2} \ge 3 - {3 + 1 \over 2} = 3 - 2 = 1,
\end{align*}
a contradiction.
\end{proof}

Recalling the definition of $n$-differentiability over $(\LC, \LC')$ (\cref{defn:differentiation}), we get the following.

\begin{cor}[Qualitative differentiability for \D-measures]\label{st:D-positive-sets-w-differentiable}
For any \D-measured semigroup $(G, \AC, \mu)$, every $\mu$-positive set is $\w$-differentiable over $(\LC_\mu, \Pos[\FC_\mu])$, and hence, $\LC_\mu = C^\w(\LC_\mu, \Pos[\FC_\mu]) \subseteq C^\w_{\FC_\mu}$. In particular, $\AC_\mu \subseteq \MeasF[\FC_\mu]$.
\end{cor}
\begin{proof}
The previous proposition implies that every set in $\LC_\mu$ is $1$-differentiable over $(\LC_\mu, \Pos[\FC_\mu])$, so, by \cref{st:self-induction}, $\LC_\mu = C^\w(\LC_\mu, \Pos[\FC_\mu])$. Because $\LC_\mu \subseteq \Pos[\FC_\mu]$, \cref{st:diff_is_hereditary-in-L} implies that $C^\w(\LC_\mu, \Pos[\FC_\mu]) \subseteq C^\w_{\FC_\mu}$.
\end{proof}

\subsection{A van der Corput lemma for \D-measures}

\cref{st:D-positive-sets-w-differentiable,st:vdC_for_del-filters} yield the following:

\begin{cor}[van der Corput lemma for \D-measures]\label{st:vdC_for_D-measures}
	For any \D-measured semigroup $(G, \AC, \mu)$, the filtered \M-semigroup $(G, \AC_\mu, \FC_\mu)$ has the van der Corput property, i.e. for every weakly upper $\AC_\mu$-semimeasurable bounded sequence $(e_g)_{g \in G}$ in a Hilbert space $\Hil$, we have
	$$
	\lim_{h \to \FC_\mu} \lim_{g \to \FC_\mu} \gen{e_g, e_{gh}} = 0 \implies \lim_{g \to \FC_\mu} \gen{f,e_g} = 0, \ \forall f \in \Hil.
	$$
\end{cor}

\begin{remark}
	If $\AC$ is a $\si$-algebra (as in \eqref{example:comeager_in_Baire_group}, \eqref{example:conull_in_amenable_group}--\eqref{example:conull_in_probability_group} of Examples \ref{examples:D-measures}), we recall from \cref{remark:equivalence_of_measurability_for_si-algebras} that weakly $\AC$-measurable sequences (in the classical sense) are weakly upper $\AC$-semimeasurable, and hence also weakly upper $\AC_\mu$-semimeasurable. Thus, the conclusion in \cref{st:vdC_for_D-measures} holds for all weakly $\AC$-measurable sequences.
\end{remark}

With this remark in mind, we see that the following particular instances of van der Corput lemmas that have appeared in the literature all follow as direct applications of \cref{st:vdC_for_D-measures}.

\begin{instances}
	\item(Furstenberg \cite{Furst_book}*{Lemma 4.9}) \emph{For the density filter (along some \Folner sequence) $\FC_\updens$ on a countable amenable group $G$, the filtered \M-semigroup $(G, \Pow(G), \FC_\updens)$ has the van der Corput property.}
	
	\item(By essentially the same proof as for the previous instance) \emph{For a locally compact Hausdorff amenable group $G$ and a finitely additive invariant probability measure $\mu$ defined on the $\si$-algebra $\AC$ of Haar measurable subsets of $G$, the filtered \M-semigroup $(G, \AC, \FC_\mu)$ has the van der Corput property.}
		
	\item (Folklore) \emph{For a compact Hausdorff group $G$ and the normalized Haar measure $\mu$ on $G$, the filtered \M-semigroup $(G, \AC, \FC_\mu)$ has the van der Corput property, where $\AC$ is the $\si$-algebra of $\mu$-measurable subsets of $G$.}
	
	\item (By the same proof as for the previous instance) \emph{For the Loeb measure $\mu$ on an ultraproduct $G$ of a sequence of groups $G_n$ equipped with a finitely additive invariant probability measure $\mu_n$, the filtered \M-semigroup $(G, \AC, \FC_\mu)$ has the van der Corput property, where $\AC$ is the $\si$-algebra of $\mu$-measurable subsets of $G$.}
	
	\item (A variant of \cite{me_prob_groups}*{Lemma 20}) Generalizing the previous two examples, \emph{for a probability group $(G, \BC, \mu)$ (as in Definition 1 of \cite{me_prob_groups}), the filtered \M-semigroup $(G, \AC, \FC_\mu)$ has the van der Corput property, where $\AC$ is the $\si$-algebra of $\mu$-measurable subsets of $G$.}
\end{instances}

This set of instances, together with Instances \ref{instances:vdC_for_del-filters}, completes the list of all previously existing van der Corput lemmas that the author is aware of.

\subsection{Further consequences}

In terms of derivation trees, \cref{st:D-positive-sets-w-differentiable}, together with parts \ref{item:self-induction=>w-branch} and \ref{item:infinite-branch=>IP-set} of Proposition \ref{prop:tree-properties}, implies the following:

\begin{cor}\label{cor:positive-delta-measure-infinite-branch}
For a \D-measured semigroup $(G,\AC,\mu)$ and $A \in \LC_\mu$, $\T[{\LC_\mu,\Pos[\FC_\mu]}]{A}$ is nonempty pruned. In particular, it has an infinite branch and hence $\De(A / \LC_\mu, \Pos[\FC_\mu])$ contains an \IP-set.
\end{cor}

This last corollary implies the following weak version of what would be a ``density Hindman theorem'' for \D-measures.

\begin{cor}
Let $(G,\AC,\mu)$ be a \D-measured semigroup. For every $\mu$-positive set $A \subseteq G$ there is an \IP-set $P = \FP(h_n)_{n \in \N}$ such that for every finite product subset $Q_N := \FP(h_n)_{n < N}$, $N \in \N$, we have
$$
\bigcap_{h \in Q_N} A h^{-1} >_{\FC_\mu} 0.
$$ 
In particular, $A$ contains shifts of arbitrarily long finite product sets; more precisely, for every $N$, there is $g \in A$ such that $A \supseteq g Q_N$.
\end{cor}


\bigskip


	
\end{document}

%% file: def.tex
\usepackage{amsthm}
\usepackage{amsmath}

\usepackage{amssymb}
\usepackage{MnSymbol}

\usepackage{bbm}

\usepackage{cancel}

\usepackage{verbatim}

\usepackage{enumerate}

\usepackage{graphics}
\usepackage{graphicx}
\usepackage{float}

\usepackage{color}

\usepackage{xspace}

\usepackage{calrsfs}

\DeclareSymbolFont{defaultmathcal}{OMS}{zplm}{m}{n}
\DeclareSymbolFontAlphabet{\mathcal}{defaultmathcal}

\DeclareSymbolFont{handwritten}{OMS}{rsfs}{m}{n}
\DeclareSymbolFontAlphabet{\handcal}{handwritten}

\usepackage[alphabetic, nobysame]{amsrefs}

\usepackage{hyperref}
\hypersetup{pdfstartview={XYZ null null 1.00}, pdfpagemode=UseNone, colorlinks,breaklinks, linkcolor=blue,urlcolor=blue, anchorcolor=blue,citecolor=blue}

\usepackage[capitalise]{cleveref}

\usepackage{geometry}
\input{"\LatexDef/def_commands.tex"}

\input{"\LatexDef/def_environments.tex"}

\definecolor{gris}{RGB}{90,90,90}

%% file: def_commands.tex
\newcommand{\N}{\mathbb{N}}
\newcommand{\Z}{\mathbb{Z}}

\newcommand{\R}{\mathbb{R}}

\newcommand{\Hil}{\mathcal{H}}

\newcommand{\0}{\mathbb{\emptyset}}
\newcommand{\w}{\infty}

\newcommand\al{\alpha}
\newcommand\be{\beta}
\newcommand\de{\delta}
\newcommand{\e}{\varepsilon}

\newcommand\si{\sigma}

\newcommand{\De}{\Delta}

\newcommand{\AC}{\mathcal{A}} 
\newcommand{\BC}{\mathcal{B}}
\newcommand{\CC}{\mathcal{C}}

\newcommand{\FC}{\mathcal{F}}

\newcommand{\LC}{\mathcal{L}}

\newcommand{\PC}{\mathcal{P}}
\newcommand{\SC}{\mathcal{S}}

\newcommand{\Pow}{\handcal{P}}

\newcommand{\Fin}[1]{{#1^{<\N}}}

\newcommand{\St}{\mathrm{Stab}}

\newcommand{\mmod}[2][]{\mathrm{{#1}mod}\;#2}

\newcommand{\del}{\partial}
\newcommand{\cl}{\overline}

\newcommand{\simdiff}{\triangle}

\newcommand{\actson}{\curvearrowright}
\newcommand{\ractson}{\curvearrowleft}

\newcommand{\conc}{{^\frown}}

\newcommand{\imp}{\Rightarrow}

\newcommand{\rimp}{\Leftarrow}

\newcommand{\shortiff}{\Leftrightarrow}

\newcommand{\Folner}{F{\o}lner\xspace}

\newcommand{\Frechet}{Fr\'{e}chet\xspace}

\newcommand{\Slawek}{S\l{}awek\xspace}

\newcommand{\Szemeredi}{Szemer\'{e}di\xspace}
\newcommand{\Todorcevic}{Todor\v{c}evi\'{c}\xspace}

\newcommand{\straighttext}[1]{\text{\textnormal{#1}}}

\newcommand{\set}[1]{\left\{ #1 \right\}}

\newcommand{\gen}[1]{\langle #1 \rangle}

\newcommand{\rest}[1]{\mathord{\downharpoonright_{#1}}}

\newcommand{\defbycases}[4][\text{otherwise}]{\left\{\begin{array}{ll}
#2 & \text{if } #3 \\
#4 & #1
\end{array}\right.}


%% file: def_environments.tex
\theoremstyle{plain}

\newtheorem{theorem}[equation]{Theorem}
\newtheorem*{theorem*}{Theorem}

\newcommand{\namedthmlabelref}[1]{\ref{#1}\xspace}
\newcommand{\namedthmref}[2]{#1 \namedthmlabelref{#2}}
\newenvironment{namedthm}[2][]{\refstepcounter{equation}\par\medskip\noindent \textbf{#2~\theequation}#1\textbf{.}\itshape\xspace}{\smallskip}
\newenvironment{namedthm*}[2][]{\par\medskip\noindent \textbf{#2}#1\textbf{.}\itshape\xspace}{\smallskip}

\crefname{prop}{Proposition}{Propositions}
\newtheorem{prop}[equation]{Proposition}
\newtheorem*{propo*}{Proposition}

\crefname{property}{Property}{Properties}

\newtheorem*{property*}{Property}

\newtheorem{lemma}[equation]{Lemma}

\crefname{cor}{Corollary}{Corollaries}
\newtheorem{cor}[equation]{Corollary}
\newtheorem*{cor*}{Corollary}

\crefname{obs}{Observation}{Observations}

\newtheorem{obs*}{Observation}

\crefname{fact}{Fact}{Facts}

\newtheorem*{fact*}{Fact}

\theoremstyle{definition}

\crefname{defn}{Definition}{Definitions}
\newtheorem{defn}[equation]{Definition}
\newtheorem*{defn*}{Definition}

\crefname{question}{Question}{Questions}

\newtheorem*{question*}{Question}

\crefname{conj}{Conjecture}{Conjectures}

\newtheorem*{conj*}{Conjecture}

\crefname{example}{Example}{Examples}
\newtheorem{example}[equation]{Example}
\newtheorem*{example*}{Example}

\theoremstyle{remark}

\crefname{remark}{Remark}{Remarks}
\newtheorem{remark}[equation]{Remark}
\newtheorem*{remark*}{Remark}

\newenvironment{remarklike*}[2][]{\par\medskip\noindent \textbf{#2}#1\textbf{.}\rmfamily\xspace}{\medskip}

\newtheorem*{claim*}{Claim}

\newenvironment{case*}[1]{\smallskip\par\noindent \textit{Case~#1}: \rmfamily}{\smallskip}

\crefname{notation}{Notation}{Notations}
\newtheorem{notation}[equation]{Notation}
\newtheorem*{notation*}{Notation}

\newtheorem{terminology}[equation]{Terminology}
\newtheorem*{terminology*}{Terminology}

\crefname{convention}{Convention}{Conventions}
\newtheorem{convention}[equation]{Convention}
\newtheorem*{convention*}{Convention}

\crefname{spec}{Speculation}{Speculations}

\newtheorem*{spec*}{Speculation}

\crefname{caution}{Caution}{Cautions}

\newtheorem*{caution*}{Caution}

\newcommand{\fntsz}[1][11]{\fontsize{#1}{#1}\selectfont}

\newenvironment{acknowledgements}[1][11]{\medskip \fntsz[#1]\begin{trivlist}
		\item[\hskip \labelsep {\textit{Acknowledgements}.}]}{\end{trivlist}\smallskip}

\creflabelformat{enumi}{#2(#1)#3}

\newcommand{\iformat}[1][\textup{(\alph{enumi})}]{
	\renewcommand{\theenumi}{\textup{#1}}
	\renewcommand{\labelenumi}{\theenumi}
}

\newenvironment{goodenum}[1][(\alph{enumi})]
	{\begin{enumerate}\iformat[#1]}
	{\end{enumerate}}

\newcommand{\exampleslabelref}[2]{\ref{#1}\labelcref{#2}\xspace}
\newcommand{\examplesref}[2]{Example \exampleslabelref{#1}{#2}}

\crefname{examples}{Examples}{Examples}
\newenvironment{examples}[1][a]{\refstepcounter{equation}\par\medskip\noindent \textbf{Examples~\theequation.} 
\vspace{.6em}
\begin{enumerate}[\bfseries (#1)]
\setlength{\itemsep}{.8em}
\rmfamily}{\end{enumerate}\smallskip}

\newenvironment{examples*}[1][a]{\refstepcounter{equation}\par\medskip\noindent \textbf{Examples.}
\vspace{.6em}
\begin{enumerate}[\bfseries (\theequation.#1)]
\setlength{\itemsep}{.8em}
\rmfamily}{\end{enumerate}\smallskip}

\newcommand{\innerenumsep}{5pt}

\crefname{subsection}{Subsection}{Subsections}

\theoremstyle{plain}